\documentclass[11pt]{article}
\usepackage[dvips]{epsfig}
\usepackage{amsmath}
\usepackage{amsfonts}
\usepackage{amssymb}
\usepackage{graphicx}
\usepackage{latexsym,amssymb,amsmath,amsthm,amsfonts}
\usepackage{psfrag}
\usepackage{bbm}
\usepackage{float}
\usepackage{latexsym}
\usepackage{enumerate}
\usepackage{color}
\usepackage{xcolor}
\theoremstyle{plain}

\newtheorem{thm}{Theorem}[section]
\newtheorem{cor}{Corollary}[section]
\newtheorem{prop}{Proposition}[section]
\newtheorem{rem}{Remark}[section]
\newtheorem{exam}{Example}[section]

\numberwithin{equation}{section}

\newcommand{\be}{\begin{equation}}
\newcommand{\ee}{\end{equation}}
\newcommand{\bi}{\bibitem}
\newcommand{\ben}{\begin{enumerate}}
\newcommand{\een}{\end{enumerate}}
\newcommand{\beq}{\begin{eqnarray}}
\newcommand{\eeq}{\end{eqnarray}}
\newcommand{\beqn}{\begin{eqnarray*}}
\newcommand{\eeqn}{\end{eqnarray*}}

\title{Sharp $L^p$-uncertainty principles on Finsler measure spaces
\footnotetext{The project is partially supported by NNSFC (Nos.12471044, 12071423) and Zhejiang Provincial NSFC (No. LZ26A010004). }
}

\author{\small{Ranran Li and Qiaoling Xia}\\
{\small {\it Department of Mathematics, School of Sciences }}\\ {\small{\it  Hangzhou Dianzi University}}\\
 {\small {\it Hangzhou, 310018, Zhejiang Province, P.R.China}}\\
 {\small {\it E-mail address: 2343070107@hdu.edu.cn; xiaqiaoling$@$hdu.edu.cn}}}

\date{}
\begin{document}
\maketitle{}

\begin{abstract}
In this paper, we prove the $L^p(p>1)$-uncertainty principles for any $1<p<n$, including the classical Heisenberg-Pauli-Weyl inequality, Caffarelli-Kohn-Nirenberg interpolation inequality and Hardy inequality in $\mathbb R^n$ as special cases,  on $n(\geq 2)$-dimensional forward complete and noncompact Finsler measure spaces $(M, F, \mathfrak{m})$ with curvatures bounded from above or below by constants. Further, we characterize the sharpness of $L^p$-uncertainty principles in terms of the reversibility of $F$ and the bounds of flag (or Ricci) curvature and S-curvature induced by the measure $\mathfrak m$ and obtained some rigidity results, which generalize the related ones in \cite{HKZ} in Finslerian case and \cite{KKPZ} in Riemannian case.

{\small{\it  MSC 2010: }} 26D10, 53C60, 53C23

{\small{\it Keywords:} Finsler measure space, uncertainty principle, comparison theorem, curvature, sharp constant.}
\end{abstract}

\section{Introduction}

Heisenberg uncertainty principle in quantum mechanics states that the position and the momentum of particles cannot be both determined explicitly but only in a probabilistic sense with a certain uncertainty (\cite{Hei}). Its rigorous mathematical formulation, which was due to Pauli and Weyl (\cite{We}), states that a function itself and its Fourier transform cannot be well localized simultaneously.  In $\mathbb R^n$, the Heisenberg-Pauli-Weyl uncertainty principle is described by the following inequality (also called the {\it Heisenberg-Pauli-Weyl inequality}):
\beq \left(\int_{\mathbb R^n}|\nabla u|^2dx\right)\left(\int_{\mathbb R^n}|x|^2u^2dx\right)\geq \frac{n^2}4\left(\int_{\mathbb R^n}u^2dx\right)^2 \label{HUP}\eeq for any  $u\in C_0^\infty(\mathbb R^n)$, in which the constant $\frac{n^2}4$ is sharp and the extremal functions are given (up to a constant) by Gaussian functions $u(x)=e^{-c|x|^2}$ for some constant $c>0$.

On the other hand, the {\it Caffarelli-Kohn-Nirenberg interpolation inequality} established in \cite{CKN} plays an important role in the theory of functional analysis and PDE. It is described in $\mathbb R^n$ by
\beq \left(\int_{\mathbb R^n}|\nabla u|^2dx\right)\left(\int_{\mathbb R^n} \frac{|u|^{2p-2}}{|x|^{2q-2}}dx\right)\geq \frac{(n-q)^2}{p^2}\left(\int_{\mathbb R^n}\frac{|u|^p}{|x|^q}dx\right)^2 \label{CKN}\eeq for any $0< q< 2< p$, $2 < n <\frac{2(p-q)}{p-2}$ and $u\in C_0^\infty(\mathbb R^n)$, where the constant $\frac{(n-q)^2}{p^2}$ is sharp and the extremal functions are given (up to a constant) by $u(x)=\left(c+|x|^{2-q}\right)^{\frac 1{2-p}}$ for some constant $c>0$. Obviously, when $p\rightarrow 2$ and $q\rightarrow 0$, (\ref{CKN}) turns to be (\ref{HUP}). When $p\rightarrow 2$ and $q\rightarrow 2$, (\ref{CKN}) turns to be the {\it Hardy inequality}, which states that the $L^2$-norm of the singular term $u(x)/|x|$ is controlled by the $L^2$-Dirichlet norm of the function $u\in C_0^\infty(\mathbb R^n)$ (\cite{HPL}, \cite{BEL}), i.e.,
\beq \int_{\mathbb R^n}|\nabla u|^2dx\geq \frac{(n-2)^2}4\int_{\mathbb R^n}\frac{u^2}{|x|^2}dx \label{Har}\eeq for any $u\in C_0^\infty(\mathbb R^n)$, where the constant $\frac{(n-2)^2}4$ is sharp but never achieved. Note that the right hand side of (\ref{CKN}) and (\ref{Har}) are large if $u$ is localized to the origin ($x=0$). Thus their momentum are large as well. Because of this, the above three inequalities are uniformly called the {\it $L^2$-uncertainty principle}.

Based on the related works in $\mathbb R^n$, the uncertainty principles (including some weighted inequalities) have been widely investigated and developed in curved spaces, such as Riemannian manifolds, Finsler manifolds and even more general metric measure spaces under some curvatures conditions (\cite{CX}--\cite{DD}, \cite{HKZ}-\cite{KLZ}, \cite{MM} etc. and references therein).  In particular, Huang-Krist\'aly-Zhao first extended the classical uncertainty principles in $\mathbb R^n$ to the Finsler setting and obtained the sharpness of the $L^2$-uncertainty principles on forward complete Finsler measure spaces with  nonpositive flag curvature and S-curvature or nonnegative Ricci curvature and S-curvature (\cite{HKZ}). On the other hand, Krist\'aly-Li-Zhao found the failure of $L^2$-uncertainty principles on forward complete Finsler measure spaces  with Ric$\geq -(n-1)k^2$ and S-curvature $\mathbf S\geq (n-1)h$ for $h>k\geq 0$ (\cite{KLZ}). A natural question is to determine the validity of $L^2$-uncertainty principles on forward complete Finsler measure spaces with Ric$\leq -(n-1)k^2$ and $\mathbf S\leq (n-1)h$ for some $h, k\in \mathbb R$, or whether there are other forms of uncertainty principles on $(M, F, \mathfrak m)$ with Ric$\geq -(n-1)k^2$ and $\mathbf S\geq (n-1)h$ for $h>k\geq 0$. In this paper, we try to answer these questions. In fact we obtain more general $L^p$-Heisenberg-Pauli-Weyl inequality, $L^p$-Caffarelli-Kohn-Nirenberg interpolation inequality and $L^p$-Hardy inequality etc., uniformly called the {\it $L^p$-uncertainty principle}, and their sharpness on forward complete and noncompact Finsler measure spaces with curvatures bounded from above or below. To state our results, let us introduce some terminologies and notations. See Section 2 for more details.

Let $(M, F, \mathfrak m)$ be a Finsler measure space equipped with a Finsler metric $F$ and a smooth measure $\mathfrak m$, and $g=(g_{ij}(x, y))$ be the fundamental tensor of $F$. If $F(x, -y)=F(x, y)$ for any $(x, y)\in TM$, then we say that $F$ is {\it reversible}. Otherwise, $F$ is said to be {\it nonreversible}. In this case, we define the {\it reversibility} of $F$  by
$$\Lambda: =\sup_{x\in M}\Lambda_x, \ \ \ \ \ {\rm where}\ \ \Lambda_x=\sup_{y\in T_x M\backslash\{0\}} \frac{F(x, -y)}{F(x,y)}.$$
Obviously, $1\leq\Lambda\leq+\infty$ and $F$ is reversible if and only if $\Lambda=1$. Riemannian metrics are a class of special reversible Finsler metrics. As in Riemannian case, the Riemannian geometric quantities uniquely determined by the metric $F$ are the flag curvature $\mathbf K$ and the Ricci curvature ${\mathbf{Ric}}$, which are natural generalizations of sectional curvature and Ricci curvature in Riemannian geometry (see \S 2 for definitions). However, the measure $\mathfrak m$ can not be uniquely determined by $F$ in Finslerian case. One of frequently used measures is the Busemann-Hausdorff measure, denoted by $\mathfrak m_{BH}$ (see \S 2.1 for definition). This means that the non-Riemannian geometric quantities related with the measure appear in Finsler geometry. One of important non-Riemannian quantities is S-curvature $\mathbf S$ introduced by Z. Shen in \cite{Sh1}, which is defined by the change rate of distorsion $\tau$ of $F$ along a geodesic, where $\tau$ is defined by
$$ \tau(x, y):=\log\left\{ \frac{\sqrt{det(g_{ij}(x, y))}}{\sigma(x)}\right\},$$ here $\sigma$ is the density function of $\mathfrak m$. For any $x\in M$, let $S_x(M)=\{y\in T_xM|F(x, y)=1\}$ be the indicatrix of $F$ at $x$ and $d\nu_x$ be the Riemannian measure on $S_xM$ induced by $F$ (cf. \S 2). Define the {\it integral of distorsion} by
\beq  %\mathcal
\mathfrak{L}_{\mathfrak{m}}(x):=\int_{S_xM}e^{-\tau(y)}d\nu_x(y), \label{tau-integ}\eeq
which is finite (\cite{HKZ}). Obviously, $\mathfrak{L}_{\mathfrak{m}}(x)$ is just the area of unit sphere $S_xM$ on a Riemannian manifold $(M, g)$. Moreover, let $p, q, \alpha\in\mathbb{R}$ and $n\in\mathbb{N}$ satisfying one of the following range conventions:
\beq\label{pqn}
	\begin{cases}(\text{I})& 1<p=q < n\text{ and }-p+1 < \alpha\leq 1;\\(\text{II}) & 1 < p < q, \ -p+1 < \alpha \ \text{ and }\ 0< q(n-p)< p(n+\alpha-1),\end{cases} \eeq   which implies that $n\geq 2$. Moreover, we set
\beqn
\mathfrak{s}_k(t):=\begin{cases}\frac{\sin(\sqrt{k}t)}{\sqrt{k}},&\mathrm{if~}k>0,\\
t,&\mathrm{if~}k=0,\\
\frac{\sinh(\sqrt{-k}t)}{\sqrt{-k}},&\mathrm{if~}k<0, \end{cases} \ \ \ \ \ \ \ \ \ \ \ \ \ \ \ \mathfrak{ct}_k(t):=\frac{\mathfrak{s}'_k(t)}{\mathfrak{s}_k(t)} \eeqn defined on $(0,\frac{\pi}{\sqrt{k}})$ for $k\in\mathbb{R}$, where we set $\frac{\pi}{\sqrt{k}}: =+\infty$ if $k\leq0$. We always denote
by $p':=\frac{p}{p-1}$ the conjugate index of $p$ throughout the paper.

For any $x\in M$ and $u\in C_0^\infty(M)\backslash\{0\}$, let $r_x:=d_F(x, \cdot)$ be the distance function from $x$ induced by $F$ and
\beqn
	U_{p, q, \alpha, k, l}^{\max}(x, u):=\frac{\left(\int_{M}\max\left\{F^{\ast p}(\pm du)\right\}d\mathfrak{m}\right)^{\frac{1}{p}} \left(\int_{M} r_x^{p'\alpha} |u|^{p'(q-1)} d\mathfrak{m}\right)^{\frac{1}{p'}}}{\left|\int_{M}\left(1+\frac{n-1}{n+\alpha-1}\zeta_{k, l}(r_x)\right)r_x^{\alpha-1}|u|^q d\mathfrak{m}\right|}, \eeqn
where  $\zeta_{k, l}(t):=t\big(\mathfrak{ct}_{k}(t)-l\big)-1$ for $k, l\in \mathbb R$ and $t\in(0,\frac{\pi}{\sqrt{k}})$. Obviously, $\zeta_{0, 0}(t)=0$. Given a smooth measure $\mathfrak m$ on $(M, F)$, denote by $m(B^+_x(r))$ the volume of forward geodesic ball $B^+_x(r)$ with radius $r$ centered at $x$ with respect to $\mathfrak m$.

\begin{thm}\label{thm12}
Let $(M, F, \mathfrak{m})$ be an $n (\geq 2)$-dimensional forward complete and noncompact Finsler measure space,  and $p, q, \alpha\in\mathbb{R}$ and $n\in\mathbb{N}$ satisfying (I) or (II) in $(\ref{pqn})$. Assume that ${\mathbf K}\leq\kappa$ and $\mathbf S\leq(n-1)h$ for some $\kappa, h\in\mathbb{R}$ satisfying either $\kappa=h=0$ or $h>0$ if $\kappa\geq 0$ and $h>\sqrt{-\kappa}>0$ if $\kappa<0$.
Then, for every $x\in M$ and $ u\in C_0^\infty(M)\backslash\{0\}$, we have
\beq 	U_{p, q, \alpha, \kappa, h}^{\max}(x,u)\geq \frac{n+\alpha-1}{q}, \label{U-max}\eeq and the constant $\frac{n+\alpha-1}{q}$ is sharp if $F$ is reversible. Further, if there is a point $o\in M$ such that
\beq \Lambda_o=\Lambda, \ \ \ \mathfrak{L}_{\mathfrak{m}}(o)=\inf_{x\in M} \mathfrak{L}_{\mathfrak{m}}(x).\label{LL}\eeq   then the following statements are equivalent.

{\rm(1)} The equality in (\ref{U-max}) holds for every $x\in M$.

{\rm(2)} The equality in (\ref{U-max}) holds at $o\in M$.

{\rm(3)} $\mathfrak m=\mathfrak m_{BH}$ and $(M, F, \mathfrak m)$ is reversible with $\mathbf K\equiv \kappa$ and  $\mathbf S\equiv (n-1)h$.

In this case,  $m(B_r^+(x))=\chi(r) \mathfrak{L}_{\mathfrak{m}}(o)$ only depending on the radius $r$ for any $x\in M$, where $\chi(r):=\int_0^r (e^{-ht}\mathfrak s_{\kappa}(t))^{n-1}dt$ and $u$ is given by \beq |u|=\left\{\begin{array}{ll} ce^{-\frac 1\beta r^{\beta}} & p, q, \alpha\ \mathrm{satisfy\  (I)}, \\
  (\beta\vartheta)^{\vartheta}(\tilde c+r^\beta)^{-\vartheta}, & p, q, \alpha\ \mathrm{satisfy\  (II)}, \end{array} \right. \label{u-func} \eeq
 where $c$,  $\beta:=\frac{p+\alpha-1}{p-1}$, $\vartheta: = \frac{p-1}{q-p}$ are positive constants and $\tilde c$ is a constant with $r^\beta+\tilde c>0$.
\end{thm}

 Similarly, for any $x\in M$ and $u\in C_0^\infty(M)\backslash\{0\}$, let
\beqn
	U_{p, q, \alpha, k, l}^{\min}(x,u):=\frac{\left(\int_{M}\min\{F^{\ast p}(\pm du)\}d\mathfrak{m}\right)^{\frac{1}{p}} \left(\int_{M} r_x^{p'\alpha} |u|^{p'(q-1)} d\mathfrak{m}\right)^{\frac{1}{p'}}}{\left|\int_{M}(1+\frac{n-1}{n+\alpha-1}\zeta_{k, l}(r_x))r_x^{\alpha-1}|u|^q d\mathfrak{m}\right|}. \eeqn

\begin{thm}\label{thm13} Let $(M, F, \mathfrak{m})$ be an $n (\geq 2)$-dimensional forward complete Finsler measure space, and $p, q, \alpha\in\mathbb{R}$ and $n\in\mathbb{N}$ satisfying the case (I) or (II) in $(\ref{pqn})$. Assume that ${\mathbf{Ric}}\geq(n-1)\kappa$ and ${\mathbf S}\geq(n-1)h$ for some $\kappa, h\in\mathbb{R}$ satisfying either $\kappa=h=0$ or $h>0$ if $\kappa\geq 0$ and $h>\sqrt{-\kappa}>0$ if $\kappa<0$. If there is some $o\in M$ such that (\ref{LL}) holds,
% and the constants $\kappa, h$ also satisfy $h>0$ when $\kappa\geq 0$ and $h>\sqrt{-\kappa}>0$ when $\kappa<0$ in the case of $\alpha>0$,
 then the following statements are equivalent.

{\rm(1)} The inequality
\beq
	U_{p, q, \alpha, \kappa, h}^{\min}(x,u)\geq \frac{n+\alpha-1}{q},\label{U-min}\eeq
 holds for every $ u\in C_0^\infty(M)\backslash\{0\}$ and $x\in M$.

{\rm(2)} The inequality (\ref{U-min}) holds at $o\in M$ for every $ u\in C_0^\infty(M)\backslash\{0\}$.

{\rm(3)} $\mathfrak m=\mathfrak m_{BH}$ and $(M, F, \mathfrak m)$ is reversible  with $\mathbf K\equiv \kappa$ and  $\mathbf S\equiv (n-1)h$.

In this case $m(B_r^+(x))=\chi(r) \mathfrak{L}_{\mathfrak{m}}(o)$ only depending on the radius $r$ for any $x\in M$, where $\chi(r)$ is given as in Theorem \ref{thm12}.
\end{thm}

 If $F=g$ is Riemannian and $\mathfrak m$ is the Riemannian measure induced by $g$, then $\mathbf S=0$ and $\mathfrak{L}_{\mathfrak{m}}(x)$ is the area of unit sphere $S_xM$ as mentioned before. In this case, (\ref{LL}) is trivial and Theorems \ref{thm12}--\ref{thm13} are reduced to the corresponding ones in Riemannian case. Moreover, $m(B_r^+(x))=\chi(r) \mathfrak{L}_{\mathfrak{m}}(o)=V_\kappa(r)$ (the volume of geodesic ball of radius $r$ in the Riemannian space form $\mathbf M_\kappa^n$), which implies that $(M, g)$ isometric of $\mathbf M_\kappa^n$ (cf. Theorems 5.1--5.2, \cite{KKPZ}). It should be noted that Theorem 5.2 in \cite{KKPZ} still holds for $-p+1<\alpha\leq 1<p<n$ instead of $0<\alpha\leq 1<p<n$ when $\kappa<0$. In Finslerian case, $m(B_r^+(x))=\chi(r) \mathfrak{L}_{\mathfrak{m}}(o)$  means that the volumes of forward geodesic balls with same radii $r$ at any $x\in M$ with respect to $\mathfrak m$ are constants $\chi(r) \mathfrak{L}_{\mathfrak{m}}(o)$, equivalently,  $\mathbf K\equiv \kappa$ and  $\mathbf S\equiv (n-1)h$ (see Corollary  \ref{cor33} below). The classical example is the Funk metric $\mathcal F$ defined on a bounded strongly convex domain $\Omega\subset \mathbb R^n$, which is forward complete. Moreover, $\mathcal F$ has constant flag curvature $\mathbf K=\kappa=-1/4$ and constant S-curvature $\mathbf S=(n+1)/2$ with respect to the Busemann-Hausdorff measure $\mathfrak m_{BH}$ (\cite{Sh}, \cite{KLZ}). Obviously ${\mathbf{Ric}}=-(n-1)/4$ and $h=(n+1)/2(n-1)$. Thus $(M, \mathcal F, \mathfrak m_{BH})$ satisfies the curvature assumptions in Theorems \ref{thm12}--\ref{thm13}. It is a question to study the topological or metric (measure) structure of $(M, F, \mathfrak m)$ with $\mathbf K\equiv \kappa$ and  $\mathbf S\equiv (n-1)h$, which will be studied later.

 When $\kappa=h=0$ and $p=2$ in Theorems \ref{thm12}--\ref{thm13}, we have $p'=2$ and $\zeta_{\kappa, h}(r)=0$. If we choose $\alpha=-\mathfrak q+1$ for $0<\mathfrak q < 2=p\leq q$, then
\beqn p, q, \alpha \ {\rm satisfy}\left\{ \begin{array}{ll} {\rm (I)}\,  &{\rm iff}\ 0<\mathfrak q< 2 = p=  q<n, \\
 {\rm (II)}  & {\rm iff}\ 0<\mathfrak q< 2=p< q \ \ {\rm{and}}\ \  2< n<\frac{2(q-\mathfrak q)}{q-2}, \end{array}\right.\eeqn  where the second case (II) is just  (II) of (1.6) in \cite{HKZ}, in which $p$ and $q$ are replaced by $q$ and $\mathfrak q$ respectively.  Thus (\ref{U-max}) and (\ref{U-min}) are reduced to
 \beq U_{2, q, -\mathfrak q+1, 0, 0}^{\max}(x,u)\geq \frac{n-\mathfrak q}{q}, \ \ \ \ U_{2, q, -\mathfrak q+1, 0, 0}^{\min}(x,u)\geq \frac{n-\mathfrak q}{q},\label{UU}\eeq which are exactly the inequalities in Theorems 1.1 and 1.3 in \cite{HKZ}. Consequently,  Theorems \ref{thm12}--\ref{thm13} are generalizations of Theorems 1.1 and 1.3 in \cite{HKZ}.
When  $\mathfrak q\rightarrow 0$ (i.e., $\alpha\rightarrow 1$),  (\ref{UU})$_1$ becomes
\beq U_{2, q, 1, 0, 0}^{\max}(x,u)\geq \frac{n}{q}. \label{UU-1} \eeq
(\ref{UU}) and (\ref{UU-1}) are respectively the Finslerian versions of the classical sharp Caffarelli-Kohn-Nirenberg inequality and Heisenberg-Paul-Weyl uncertainty principle in $\mathbb R^n$ (\cite{CKN}, \cite{We}), which were first proved in \cite{HKZ} in a different way.

As $q\rightarrow p$ and $\alpha\rightarrow -p+1$ in (\ref{U-max}), we get the $L^p$-Hardy inequality (\ref{H-Ineq}) below. In fact,  we obtain (\ref{H-Ineq}) under more weaker curvature conditions (also see Proposition \ref{prop45}).
\begin{thm}  \label{thm14}
Let $(M,F,\mathfrak{m})$ be an $n (\geq 2)$-dimensional forward complete and noncompact Finsler metric measure space. Assume that $\mathbf K\leq\kappa$ and $\mathbf S\leq(n-1)h$, where $\kappa, h\in\mathbb{R}$. Then, for any $1<p<n$, $x\in M$ and $u\in C_0^\infty(M)\backslash\{0\}$, we have
\beq\label{H-Ineq}
	U_{p,p,1-p, \kappa, h}^{\max}(x,u)\geq \frac{n-p}{p}.
\eeq
 Moreover, if $F$ is reversible, then the constant $\frac{n-p}{p}$ is sharp but never achieved.
\end{thm}

\begin{rem} {\rm When $p=2$ and $\kappa=h=0$, Theorem \ref{thm14} is reduced to Theorem 1.4 in \cite{HKZ}. If $\kappa=-k^2\leq 0$  in Theorem \ref{thm12}, then ${\mathbf {Ric}}\leq -(n-1)k^2$. In this case, Theorems \ref{thm12} and \ref{thm14} show that the $L^p$-uncertainty principles are true for any $1<p<n$. This gives an answer for validity of $L^p$-uncertainty principles under the curvatures bounded from above mentioned in the introduction. If ${\mathbf{Ric}}\geq(n-1)\kappa=-(n-1)k^2$ and ${\mathbf S}\geq(n-1)h$ for $k, h\in\mathbb{R}$ satisfying $h>k\geq 0$, then Theorem \ref{thm13} shows that the $L^p$-uncertainty principles are still true although some $L^2$-uncertainty principles fail in Theorem 1.1 of \cite{KLZ}.  It is worth mentioning that $\zeta_{k, h}\neq 0$ when $h>k\geq 0$. The inequality (\ref{U-min}) when $p=2$ is different from the ones considered in \cite{KLZ}. In this sense, Theorem \ref{thm13} and Theorem 1.1 in \cite{KLZ} hold independently.}
\end{rem}

It is surprising that all inequalities in Theorems \ref{thm12}-\ref{thm14} are from a generic functional inequality in Theorem \ref{thm11} in Section 4 and applications of some comparison theorems in Section 3. In this spirit, we hope the interested readers to build further functional inequalities through Theorems \ref{thm11}. As examples, we list $L^p$-Caccioppoli inequality and $L^p$-spectral gap estimate on Finsler measure spaces in Appendix.

\section{Preliminaries}

In this section, we recall some definitions and notations in Finsler geometry. See \cite{BCS}, \cite{Sh} or \cite{Xia} for more details.
\subsection{ Finsler metrics and measures}

Let $M$ be be an $n$-dimensional differential manifold and $TM$ (resp. $T^{\ast}M$) be the tangent (resp. cotangent) bundle on $M$. If there is a continuous function $F: TM\rightarrow [0, \infty)$ satisfying (1) $F\in C^\infty(TM\backslash\{0\})$, (2) $F(x, \lambda y)=\lambda F(x, y)$ on $TM$ for any $\lambda>0$, (3) the matrix $(g_{ij})=(\frac 12 [F^2]_{y^iy^j})$ is positive definite, then $F$ is called a {\it Finsler metric} on $M$ and the pair $(M, F)$ is called a {\it Finsler manifold}. The tensor $g:=g_{ij}(x, y)dx^i\otimes dx^j$ is called the {\it Fundamental tensor} of $F$. Given a smooth measure $\mathfrak{m}$, the triple $(M, F, \mathfrak{m})$ is called a {\it Finsler metric measure space} (simply, {\it Finsler measure space}). In particular, if  $F$ satisfies $F(x, y)=F(x, -y)$ for any $(x, y)\in TM$, then $F$ is said to  be {\it reversible}. Otherwise, we say that $F$ is {\it nonreversible}. In this case, we define the {\it reversibility} $\Lambda$ of $F$ as in the introduction.

The {\it dual metric} $F^{\ast}: T^{\ast}M\rightarrow \mathbb R$ of $F$ is defined by
$$F^{\ast}(x, \xi):=\sup\limits_{(x, y)\in TM\backslash\{0\}}\frac{\xi(y)}{F(x,y)},$$
which is also a Finsler metric on $M$. Similarly, we may define the fundamental tensor $g^{\ast}=(g^{\ast ij}(x, \xi))$ of $F^{\ast}$, where $g^{{\ast}ij}=\frac 12[F^{{\ast}2}]_{\xi^i\xi^j}$ and the reversibility $\Lambda^{\ast}$ of $F^{\ast}$ by
$$\Lambda^{\ast}: =\sup_{x\in M}\Lambda^{\ast}_x, \ \ \ \ \ {\rm where}\ \ \Lambda^{\ast}_x=\sup_{\xi\in T^{\ast}_x M\backslash\{0\}} \frac{F^{\ast}(x, -\xi)}{F^{\ast}(x, \xi)}.$$ Then $\Lambda_x=\Lambda^{\ast}_x$ for any $x\in M$ and hence $\Lambda=\Lambda^{\ast}$ (\cite{ZX}, \cite{HKZ}, \cite{Xia}).

Given a smooth vector field $V$ on $M$, we define the {\it weighted Riemannian metric} $g_V$ on $(M, F)$ by $g_V=g_{ij}(x, V)dx^i\otimes dx^j$, equivalently,
\beq g_V(X,Y)=g_{ij}(x, V)X^iY^j, \label{gV}\eeq for any vector fields $X, Y$ on $M$. Thus, we have $F^2(V)=g_V(V, V)$.

For $x_1, x_2\in M$, the {\it distance} $d_F$ from $x_1$ to $x_2$ induced by $F$ is defined by
\beqn d_F(x_1, x_2):=\inf_{c}\int_0^1F(\dot{c}(t))dt, \eeqn where the infimum is taken over all $C^1$ curves $c: [0,1]\rightarrow M$ such that $c(0)=x_1$ and $c(1)=x_2$. In general, $d_F(x_1, x_2)\neq d_F(x_2, x_1)$  unless $F$ is reversible. If $\Lambda<\infty$, then $d_F(x_1, x_2)$ and $d_F(x_2, x_1)$ are comparable. Let
\beqn B^+_R(x):=\{z\in M\mid d_F(x,z)<R\}, \ \ \  B^-_R(x):=\{z\in M\mid d_F(z, x)<R\}. \eeqn They are respectively called the {\it forward geodesic ball} and {\it backward geodesic ball} on $(M, F)$. If $F$ is reversible, these two geodesic balls coincide.

A $C^1$ curve $\gamma: [0, \ell]\rightarrow M$ is called a {\it geodesic} if it has constant speed (i.e.,
$F(\gamma, \dot \gamma)$ is constant) and if it is locally minimizing. A geodesic is said to be {\it normal} if $F(\dot\gamma)=1$.
 % It can be characterized by $\ddot{c}^i(t)+2G^i(c(t), \dot c(t))=0$, here  $G^i=G^i(x,y)$ are the geodesic coefficients given by
%\beq G^i=\frac{1}{4}g^{ij}\left\{[F^2]_{x^ky^j}y^k-[F^2]_{x^j}\right\}.\label{Gi}\eeq
Such a geodesic is in fact a $C^\infty$ curve. For any $x\in M$ and $y\in T_xM$, there is a unique minimal geodesic $\gamma: (-\epsilon, \epsilon)\rightarrow M$ with $\gamma(0)=x$ and $\dot\gamma(0)=y$ for a small $\epsilon>0$. For any $y\in S_{x}M:=\{y\in T_{x}M|F(x,y)=1\}$,
the {\it cut value} $i_{y}$ of $y$ is defined by
$$i_{y}:=\sup\{r\in \mathbb R^+|\text{the segment} \ \gamma_{y}|_{[0,r]} \ \text{is globally minimizing}\}.$$ For any $x\in M$, let
$$\mathfrak{i}_{x}:=\inf_{y\in S_{x}M}i_{y}, \ \ \ \ \ \ \mathfrak{i}_{M}:=\inf_{x\in M}\mathfrak i_{x}.$$ We call $\mathfrak{i}_{x}$ and $\mathfrak{i}_{M}$
the {\it injectivity radius} at $x$  and the {\it injectivity radius} of $M$ respectively. The {\it cut locus} of $x$ is given by
$$\mathcal C_{x}:=\{\gamma_{y}(i_{y})|\text{$y\in S_{x}M$ \ \text{with} \ $i_{y}<\infty$}\},$$
which is closed and has null measure. A Finsler manifold $(M, F)$ is said to be {\it forward complete} (resp. {\it backward complete}) if each geodesic defined on $[0, l)$ (resp. $(-l, 0]$) can be extended to a geodesic defined on $[0, \infty)$ (resp. $(-\infty, 0]$). These two kind of completeness are equivalent when $\Lambda<\infty$. In this case, we say that $(M, F)$ is {\it complete}.

 In general, the measure on a Finsler manifold $(M, F)$ can not be uniquely determined by the Finsler metric $F$. One of the frequently used measures on $(M, F)$ is  the {\it Busemann-Hausdorff measure} $m_{BH}$, whose volume form is defined by $$d\mathfrak m_{BH}=\frac {\mbox{Vol}(\mathbb B^n(1))}{\mbox{Vol}\{(y^i)\in \mathbb R^n|F(y)<1\}}dx, $$ where $\mathbb B^n(1)$ is the Euclidean unit sphere and Vol$(\cdot)$  denotes the Euclidean volume. It is reduced to the Riemannian measure if $F$ is  Riemannian.
Fixed a point $o\in M$,  the distance function $r(\cdot):=d_F(o, \cdot)$ is smooth on $M\backslash\{\mathcal C_o\cup \{o\})$. Thus there is a polar coordinate system $(r, y)$ around $o$ such that $r:= d_F(o, \cdot)<\mathfrak i_o$  and $y\in S_oM$. In the polar coordinate $(r, y)$,  rewrite $d\mathfrak{m}= \sigma_o(r, y)dr\wedge d\Theta$, where $d\Theta$ is the standard volume form on $S_o(M)$. Note that each Finsler metric $F$ induces a Riemannian metric $g_x=g_{ij}(x, y)dy^i\otimes dy^j$  on $T_xM\backslash\{0\}$ for any $x\in M$. Let $\dot g_o=(\dot{g}_{ij}(o, y))$ and $d\nu_o$ be respectively the Riemannian metric and the Riemannian volume form on $S_o(M)$ induced by $g_o$.  Then
\beqn d\nu_o(y)=\sqrt{\det(g_{ij}(o, y))}\left(\sum\limits_{i=1}^n(-1)^{i-1}y^idy^1\wedge\cdots\wedge \hat{dy^i} \wedge\cdots\wedge dy^n\right).\eeqn
Thus  $d\mathfrak m$ is reexpressed by
\beq\label{dm-polar}
	d\mathfrak{m} = \hat{\sigma}_o(r, y) dr \wedge d\nu_o(y), \eeq
where $\hat{\sigma}_o(r, y)=\frac{\sigma_o(r, y)}{\sqrt{\det \dot g_o(y)}}$ with
\beq
	\lim_{r \to 0^+} \frac{\hat{\sigma}_o(r, y)}{r^{n-1}} = e^{-\tau(y)} \label{sigma-limit}\eeq
(Lemma 3.1, \cite{ZS}). With this, we define the {\it integral of distortion} by (\ref{tau-integ}), which is finite.
 We say that two smooth measures $\mathfrak m_1$ and $\mathfrak m_2$ on a Finsler manifold are {\it equal} or {\it equivalent}, denoted by  $\mathfrak m_1=\mathfrak m_2$, if $\mathfrak m_1$ and $\mathfrak m_2$ are equal up to a positive constant $c$.
 \begin{prop} [\cite{HKZ}] \label{prop21} Let $\mathfrak m$ be a smooth measure on a Finsler manifold $(M, F)$. Then $\mathfrak m=\mathfrak m_{BH}$ if and only if $ \mathfrak{L}_{\mathfrak{m}}$ is constant on $M$. \end{prop}

\subsection{Connection and curvatures}

 Let $\pi: TM\backslash\{0\}\rightarrow M$ be the projective map. The pull-back $\pi^{\ast}TM$ admits a unique linear connection, which is called the {\it Chern connection}. Given a vector $V=V^i\frac {\partial}{\partial x^i}\in T_x(M)\backslash\{0\}$, the Chern connection $D$ is determined by the following equations
\beq & &\ \ \ \ \ \ \ \ \ \ \ \ \ \ \ D^V_XY-D^V_YX=[X, Y], \label{torsionfree}\\
& &Zg_V(X, Y)=g_V(D_Z^VX,Y)+g_V(X, D_Z^VY)+C_V(D_Z^VV, X,Y), \label{compatibility}\eeq  where  $X, Y, Z\in \Gamma(TM)$ (i.e., the set of smooth sections on $TM$), $g_V$ is defined by (\ref{gV}) and
\beqn C_V(X, Y, Z):=C_{ijk}(V)X^iY^jZ^k=\frac 14\frac{\partial^3F^2(x, V)}{\partial V^i\partial V^j\partial V^K}X^iY^jZ^k \eeqn
is the {\it Cartan tensor} of $F$.  $D^V_XY$ is the {\it covariant derivative} with respect to  the reference vector $V$. Note that $C_V(V, X,Y)=0$ from the homogeneity of $F$.  In terms of the Chern connection, a geodesic $\gamma$ satisfies $D_{\dot{\gamma}}^{\dot{\gamma}}{\dot{\gamma}}=0$.

 For any  nonzero smooth vector field $V$, let $\Pi:={\rm{span}}\{V, W\}$ be a tangent plane (called the {\it flag}) in $T_xM$ , the {\it flag curvature} of $F$ with the flag pole $V$ is defined by
\beqn\mathbf{K}(x, \Pi,  V)=\frac{g_V(R^V(V, W)W, V)}{g_V(V,V)g_V(W, W)-g_V(V, W)^2},\eeqn where $R^V$ is the {\it Riemannian curvature tensor} defined by
\beqn R^V(X, Y)Z:=D^V_XD^V_Y Z-D^V_Y D^V_X Z-D^V_{[X, Y]}Z.\eeqn $R^V$ (hence $\mathbf K(V, \cdot)$) is independent of the choice of connection $D$ and uniquely determined by the Finsler metric $F$ (\cite{Sh}, \cite{Xia}).
We further define the Riemannian curvature tensor $R_V$ by
$$ R_V(U):=R^V(V, U)U, \ \ \ \ {\rm{for\ any}}\ U\in \Gamma(TM). $$ It is easy to see that $R_V(V)=0$. $F$ is said to be of {\it scalar flag curvature} (resp. {\it constant flag curvature}) if $\mathbf K=\mathbf K(x, V)$ is only a function of $(x, V)$ (resp. $\mathbf K$ is a constant). It is easy to check that $F$ is said to be of {\it scalar flag curvature} $\mathbf K=K(x, V)$ if and only if \beq R_V(U)=\mathbf K\big(g_V(V, V)U-g_V(V, U)V\big), \label{flag-curv}\eeq for any $U\in \Gamma(TM)$. In particular,  $\mathbf K(x, V)=c$ (constant) in the direction $V$ if and only if $R_V(U)=c U$ for any vector field $U$ with $g_V(V, V)=1$ and $g_V(V, U)=0$.

Let $\gamma: [0, \ell]\rightarrow M$ be a unit speed geodesic and $T(t):=\dot\gamma(t)$. A vector field $J$ along $\gamma$ is called a {\it Jacobi field} if it satisfies the following equation
$$D^T_TD_T^TJ+R^T(T, J)T=0.$$
For any $C^1$ vector fields $X$ and $Y$ along $\gamma$ , the {\it index form} $I_\gamma: =I_\gamma (X, Y)$ is defined by
\beq I_\gamma(X, Y) =\int_0^\ell\left\{g_T\left(D_T^TX, D_T^TY\right)-g_T\left(R^T(X, T)T, Y\right)\right\}dt.\label{index}\eeq

The {\it Ricci curvature} is defined by
\beqn\mathbf {Ric}(x, V):=\sum_{i=1}^{n-1}\mathbf {K}( x, V, e_i),\eeqn where $e_1, \cdots, e_{n-1}, \frac V{F(V)}$ form the orthonormal basis of $T_xM$ with respect to $g_V$. We say that {\it $\mathbf {Ric}\geq (\leq) (n-1)k$} if $\mathbf {Ric}(x, V)\geq (\leq) (n-1)kF^2(x, V)$  for all $(x, V)\in TM\backslash\{0\}$.

 Let $\gamma$ be a geodesic with $\gamma(0)=x$ and $\dot{\gamma}(0)=y\in T_xM\setminus\{0\}$. The {\it S-curvature} $\mathbf S$ is defined as a rate of change of the distortion $\tau$ along the geodesic $\gamma(t)$, i.e., $$\mathbf S(x, V)=\frac d{dt}\tau(\gamma(t), \dot\gamma(t))|_{t=0}.$$
 % For the Busemann-Hausdorff measure $\mathfrak m_{BH}$, the S-curvature $\mathbf S_{BH}\equiv 0$ on a Berwald space $(M, F)$ (Proposition 7.3.1, \cite{Sh}).
 The S-curvature is a non-Riemannian geometric quantity since it vanishes on a Riemannian manifold. We say that $\mathbf S\geq h$ (resp. $\mathbf S\leq h$) for some $h\in \mathbb R$ if
 \beqn \mathbf S(x, V)\geq hF(x, V) \ \ ({\rm{resp.}}\ \mathbf S(x, V)\leq hF(x, V))\eeqn for any $(x, V)\in TM\backslash\{0\}$ (\cite{Sh}, \cite{Xia}).

\subsection{Gradient and Laplacian}

Let $(M, F)$ be an $n$-dimensional Finsler manifold. Fix a point $x\in M$, the {\it Legendre transformation} ${\mathcal L}^{\ast} : T_x^{\ast} M\to T_xM$ assigns to each covector $\xi\in T_x^{\ast} M$ the unique vector $y\in T_xM$, which minimizes the map
$$E^{\ast}(y):=\langle\xi, y\rangle-\frac{1}{2}F^2(y),$$
where $\langle\xi, y\rangle:=\xi(y)$ denoted the canonical pairing between $T_xM$ and $T^{\ast}_xM$.
The minimizer $y={\mathcal L}^{\ast}(\xi)$ can also be interpreted as the unique vector with
\beq
 F({\mathcal L}^{\ast}(\xi))=F^{\ast}(\xi)\quad\mathrm{and}\quad\langle\xi, {\mathcal L}^{\ast}(\xi)\rangle=F({\mathcal L}^{\ast}(\xi))F^{\ast}(\xi),\label{FF*}
\eeq which implies that $\langle\xi, {\mathcal L}^{\ast}(\xi)\rangle=F^{{\ast}2}(\xi)$ for any $\xi\in T_x^{\ast}M.$ Similarly, we can define the Legendre transformation ${\mathcal L}: T_xM\to T_x^{\ast}M$, assigning to each vector $y\in T_xM$ the unique covector $\xi\in T_x^{\ast} M$, which minimizes the map
$E(\xi):=\langle\xi, y\rangle-\frac{1}{2}F^{\ast 2}(\xi).$
For the minimizer $\xi={\mathcal L}(y)$, we have similar relations in (\ref{FF*}) and
\beq \mathcal L^{\ast}(\xi)=F^{\ast}(\xi) \frac{\partial F^{\ast}(\xi)}{\partial \xi}, \ \ \ \forall \ \xi\in T_x^{\ast}M.\label{F*-xi}\eeq
 By identifying $T_xM$ with $(T_x^{\ast}M)^{\ast}$ and $F$ with $F^{\ast\ast}$, we have $\mathcal L^{-1}=\mathcal L^{\ast}$. ${\mathcal L}$ and ${\mathcal L^{\ast}}$ can be extended to $TM$ and $T^{\ast}M$ in natural way (\cite{Sh}, \cite{Xia}).

For a smooth function  $u: M\rightarrow \mathbb R$, the {\it gradient} vector $\nabla u$ of $u$ is defined by $\nabla u:={\mathcal L}^{\ast}(du)\in TM$.
 Obviously, $\nabla u=0$ if $du=0$. In general, $\nabla u$ is only continuous on $M$, but smooth on $M_u:=\{x\in M|du(x)\neq 0\}$.
For the distance function $r(x)=d_F(o, x)$, we have $F(\nabla r)=F^{\ast}(dr)=1$  a.e. on $M$.
 Let $\mathfrak m$ be an arbitrary measure on $M$ and $\varphi$ be a piecewise $C^1$ function on $M$ such that every $\varphi^{-1}(t)$ is compact. The (area) measure on $\varphi^{-1}(t)$ is defined by $dA:= \nabla \varphi\rfloor d\mathfrak m$. For any $f\in C^0(M)$, we have the {\it co-area formula} (\cite{Sh})
 \beq \int_M fF(\nabla\varphi)d\mathfrak m=\int_{-\infty}^{\infty}\left(\int_{\varphi^{-1}(t)}fdA\right)dt.\label{co-area}\eeq

For any $p>1$, we define the {\it Sobolev space} $\mathcal W_0^{1, p}(M)$ as the closure of the space $C_0^\infty(M)$  with respect to the backward topology induced by the pseudo--norm
 \beqn \|u\|_{1, p}:=\|u\|_{L^p}  +  \|F^{\ast}(du)\|_{L^p},\label{norm1}\eeqn namely,  $u\in \mathcal W_0^{1, p}(M)$ if and only if there is a sequence $\{u_n\}\subset C_0^\infty(M)$ satisfying $\lim_{n\rightarrow +\infty}\|u-u_n\|_{1, p}=0$. Similarly, we also can define $\mathcal W^{1, p}(M)$.
  Note that $\mathcal W^{1, p}(M)$ (resp. $\mathcal W_0^{1, p}(M)$) is different from $ W^{1, p}(M)$ (resp. $W_0^{1, p}(M)$) defined as in \cite{Xia1}--\cite{Xia2}, which is a completion of $C_0^\infty(M)$  with respect to the Sobolev norm
\beqn \|u\|_{1, p}:=\|u\|_{L^p}  +  \|F^{\ast}(du)\|_{L^p}+ \|F^{\ast}(-du)\|_{L^p}.\label{norm1}\eeqn The latter is a linear space but the former is not a linear space unless $\Lambda<+\infty$ (Proposition 3.1, \cite{KLZ}).

 For a weakly differentiable vector field $V$ on $(M, \mathfrak m)$,  we define the {\it divergence} of $V$  by
\beqn
\int_{M}\phi \text{div} (V) d\mathfrak{m}=-\int_{M} d\phi(V) d\mathfrak{m}, \ \ \ \forall\varphi\in C_0^\infty(M). \label{div}\eeqn From this, for any $f\in C^1(M)$ and a weakly differentiable vector field $V$ on $M$, we have
\beq d(fV)=df(V) +f\text{div} V \label{fV}\eeq in the weak sense. The {\it Finsler Laplacian} of $u\in W^{1, 2}(M)$ is formally defined by
 $\Delta u:=\mbox{div}(\nabla u).$ To be more precise, $\Delta u$ should be understood in the weak sense through the identity
\beq\int_M\varphi\Delta u d\mathfrak m=-\int_M d\varphi(\nabla u)d\mathfrak m, \label{Laplacian}\eeq
for all $\varphi\in C_0^\infty(M)$. If $\int_M d\phi(\nabla u)dm=0$ for any $\phi\in C_0^\infty(M)$, then $u$ is called the {\it Finsler harmonic function}.
Locally, write $dm=\sigma(x)dx$. Then
\beq\Delta u=\frac 1{\sigma}\frac{\partial}{\partial x^i}\left(\sigma g^{ij}(\nabla u)\frac{\partial u}{\partial x^j}\right) \ \ \ {\rm{on}}\ \  M_u, \label{loc-Lap}\eeq  where $M_u:=\{x\in M|du(x)\neq 0\}$. From this, it is easy to check that
\beq \Delta u={\rm{tr}}_{\nabla u}\nabla^2u-\mathbf S(\nabla u) \ \ \  {\rm{on}}\  M_u, \label{Delta u}\eeq  where tr$_{\nabla u}$ means taking the trace with respect to the weighted Riemannian metric $g_{\nabla u}$ (Lemma 3.3 in \cite{WX}; Lemma 6.1.1 in \cite{Xia}). For definition of (first) $p$-eigenvalue for (Finsler) $p(>1)$-Laplacian, we refer to \cite{GS} and \cite{Xia1}--\cite{Xia2}, in which the energy functional is defined on $W_{loc}^{1, p}$. Similarly, we can consider the energy functional on  $\mathcal W_{loc}^{1, p}(M)$ to define the (first) $p$-eigenvalue for $p$-Laplacian (cf. \cite{WX}, \cite{KLZ}).

\section{Comparison theorems}
In this section we shall derive some sharp Laplacian comparison theorems for the distance function and volume comparison Theorems under some curvature conditions, which will be used in the subsequent arguments.

 Let $(M, F)$ be a simply connected forward complete Finsler manifold and $r: =d_F(o, \cdot)$ be the distance function from some $o\in M$. Then $r\in {\rm{Lip}}(M)$ (the set of Lipschitz functions on $M$)  and $F(\nabla r)=1$ a.e. on $M$.  If the flag curvature $\mathbf K\leq \kappa\leq 0$, then $r$ is smooth  for $0 < r \leq +\infty$ by Cartan-Hadamard's Theorem; if the Ricci curvature $\mathbf{Ric}\geq (n-1)\kappa>0$, then $r$ is smooth for $0 < r \leq \min\{\mathfrak{i}_{o}, \pi/\sqrt{\kappa}\}$ by Myers' Theorem (\cite{BCS}, \cite{Sh}, \cite{Xia}). Because of this, we use the convention: $\pi/\sqrt{\kappa}=+\infty$ when $\kappa\leq 0$.

\begin{thm}\label{thm21}
Let $(M, F, \mathfrak{m})$ be an $n$-dimensional forward complete Finsler measure space and $r=d_F(o, \cdot)$ be the distance function from some $o\in M$. If $\mathbf K(\nabla r, \cdot)\leq\kappa$ and $\mathbf S(\nabla r)\leq(n-1)h$ for some $\kappa, h\in\mathbb{R}$, then
\beq
	\Delta r\geq(n-1)\big(\mathfrak{ct}_\kappa(r)-h\big), \ \ \ \ 0 < r \leq \min\{\mathfrak{i}_{o}, \pi/\sqrt{\kappa}\}, \label{Delta-r1} \eeq in which the equality holds if and only if
\beq \mathbf K(\nabla r,\cdot)=\kappa, \ \ \ \ \mathbf S(\nabla r)=(n-1)h. \label{KS}\eeq

Further, the function
\beqn
	f_1(r):=\frac{m(B^+_r(o))}{\mathcal{V}_{o,\kappa, h, n}(r)}, \ \ \ \  0 < r \leq \min\{\mathfrak{i}_{o}, \pi/\sqrt{\kappa}\},\eeqn
is non-decreasing in $r$, where $\mathcal{V}_{o,\kappa, h, n}(r)=\chi(r) \mathfrak{L}_{\mathfrak{m}}(o)$, here $\chi(r)=\int_0^r (e^{-ht} \mathfrak{s}_\kappa(t))^{n-1}dt$.  %\beq \mathcal{V}_{o,\kappa,h,n}(r)=\int_{S_oM} e^{-\tau(y)}d\nu_o(y) \int_0^r (e^{-ht} \mathfrak{s}_\kappa(t))^{n-1}dt.\label{mathcal-V}\eeq
 In particular, $f_1(r)\geq1$ and hence $m(B_r^+(o))\geq \mathcal{V}_{o,\kappa, h, n}(r)$, in which the equality holds if and only if (\ref{KS}) holds.
In this case,
\beq \hat{\sigma}_o(r, y)=e^{-\tau(y)}(e^{-hr}\mathfrak{s}_{\kappa}(r))^{n-1},\quad \, 0 < r \leq \min\{\mathfrak{i}_{o}, \pi/\sqrt{\kappa}\}.\label{hat-sigma}\eeq \end{thm}
\begin{proof} Let us first prove the first claim. For any $x\in M$,  let $\gamma: [0, \ell]\rightarrow M$ be a normal minimal geodesic from $o=\gamma(0)$ to $x=\gamma(\ell)$ such that $x$ is not the cut point of $o$. Then $r(\gamma(t))=d_F(o, \gamma(t))$ is smooth for $t\in [0, \ell]$ and $T(t): =\dot\gamma(t)=\nabla r(\gamma(t))$ with $F(T)=1$. Since $x$ is not the cut point of $o$, we have $\nabla^2r(T, T)=\nabla^2r(T, X)=0$ for any $X\in T_xM$ by Propositions 6.3.1 in \cite{Xia}. We now consider $X\in T_xM$ with $g_{T}(T, X)=0$. Then there is a unique Jacobi field $J$ on $M$ along $\gamma$ such that $J(0)=0$ and $J(r)=X$ (note that $J=0$ if $X=0$ by uniqueness of $J$).  By Proposition 6.3.2 and its proof in \cite{Xia}, we have $g_T(T, J)=0$ along $\gamma$ and $ \nabla^2r(X, X)=I_\gamma(J, J)$.

On the other hand, let $(\tilde M, \tilde F)$ be a Finsler manifold of constant flag curvature $\kappa$. Similarly, let $\tilde\gamma: [0, \ell]\rightarrow \tilde M$ be a normal minimal geodesic such that $\tilde \gamma(\ell)$ is not the cut point of $\tilde \gamma(0)$. We denote by $\tilde {\cdot}$ the corresponding notations on $\tilde M$, for example, $\tilde r=d_{\tilde F}(\tilde\gamma(0), \cdot)$ etc.. Then $\tilde T(t): =\dot{\tilde \gamma}=\tilde \nabla \tilde r$ is smooth  with $\tilde F(\tilde T)=1$. Since $\tilde \nabla^2\tilde r(\tilde T, \cdot)=0$ as above, it suffices to consider $\tilde X\in T_{\tilde \gamma(\ell)}\tilde M$ with $\tilde g_{\tilde T}(\tilde T(\ell), \tilde X)=0$. In this case, there is a unique Jacobi field $\tilde J$ along $\tilde \gamma$ with $\tilde J(0)=0$ and $\tilde J(\ell)=\tilde X$. Thus, $\tilde g_{\tilde T}(\tilde T, \tilde J)=0$ along $\tilde \gamma$ and
\beq \tilde \nabla^2\tilde r(\tilde X, \tilde X)=\mathfrak{ct}_\kappa(\tilde r)\left\{\tilde g_{\tilde T}\big(\tilde X, \tilde X\big)-\tilde g_{\tilde T}\big(\tilde T(\ell), \tilde X\big)^2\right\}= \mathfrak{ct}_\kappa(\tilde r)\tilde g_{\tilde T}\big(\tilde X, \tilde X\big)\label{t-H-r}  \eeq by (6.3.8) in \cite{Xia}. Choose a parallel $g_{\dot\gamma}$-orthonormal frame fields $e_1(t), \ldots, e_n(t)=\dot\gamma(t)$ along $\gamma$ and write $$J(t)=\sum\limits_{i=1}^{n-1}\lambda^i(t)e_i(t), \ \ \ \ \ \lambda^i(0)=0\ (1\leq i\leq n-1).$$
Similarly, we choose a parallel $\tilde g_{\dot{\tilde \gamma}}$--orthonormal frame fields $\tilde e_1(t), \ldots, \tilde e_n(t)=\dot{\tilde\gamma}$ along $\tilde\gamma$ such that
$$\tilde X=\tilde J(\ell)=\sum\limits_{i=1}^{n-1}\lambda^i(\ell)\tilde e_i(\ell).$$ Define a vector field $Z$ along $\tilde\gamma$ by $Z(t)=\sum_{i=1}^{n-1}\lambda^i(t)\tilde e_i(t).$ Then $\tilde g_{\tilde T}(\tilde T, Z)=0$,  $Z(0)=J(0)=\tilde J(0)=0$ and $Z(\ell)=\tilde X=\tilde J(\ell)$. It follows from the index lemma (\cite{BCS}, \cite{Sh}, \cite{Xia}) and $\mathbf K(T, \cdot)\leq \kappa$ that
\beq \tilde\nabla^2\tilde r(\tilde X, \tilde X)|_{\dot{\tilde\gamma}(\ell)}&=&I_{\tilde\gamma}(\tilde J, \tilde J)\leq I_{\tilde\gamma}(Z, Z)\nonumber \\
&=& \int_0^\ell\left\{ \tilde g_{\tilde T}\left(\tilde D_{\tilde T}^{\tilde T}Z, \tilde D_{\tilde T}^{\tilde T}Z\right)-\tilde g_{\tilde T}\left(\tilde R_{\tilde T}(Z), Z\right)\right\}dt\nonumber \\
%&=&\int_0^\ell\left\{\sum_i(\lambda^{i'})^2-\kappa\sum_i(\lambda^i)^2\right\}dt \nonumber \\
&\leq & \int_0^\ell\left\{  g_{T}\left(D_{T}^{T} J, D_{T}^{T}J\right)- g_{T}\left(R_{T}(J), J\right)\right\}dt\nonumber \\
&=&I_{\gamma}(J, J)= \nabla^2 r\left(X, X\right)|_{\gamma(\ell)}.\label{Hr-XX}\eeq
Since $F(\dot\gamma)=\tilde F(\dot{\tilde\gamma})=1$, we have $r(\gamma(t))=\tilde r(\tilde \gamma(t))$ for any $t\in [0, \ell]$. Observe that (\ref{t-H-r}) keeps invariant up to a scaling of $\tilde X$. Thus we may assume that $\tilde X\in T_{\tilde \gamma(\ell)}M$ with $\tilde g_{\tilde T}(\tilde X, \tilde X)=1$ and $\tilde g_{\tilde T}(\tilde T(\ell), \tilde X)=0$, which mean that $\tilde \nabla^2\tilde r\big(\tilde X, \tilde X\big)|_{\tilde \gamma(\ell)}=\mathfrak{ct}_\kappa(\tilde r(\tilde \gamma(\ell))=\mathfrak{ct}_\kappa(r(\gamma(\ell))$ by (\ref{t-H-r}).
From this and (\ref{Hr-XX}), one obtains that $ \nabla^2r(X, X)|_x\geq \mathfrak{ct}_\kappa(r(x))$. Taking the trace on both sides of this inequality yields
\beq {\rm{tr}}_{\nabla r}(\nabla^2r)|_x=\sum\limits_{i=1}^n\nabla^2r(e_i(\ell), e_i(\ell))\geq (n-1)\mathfrak{ct}_\kappa(r(x)),\label{H(r1)}\eeq which together with (\ref{Delta u}) yields (\ref{Delta-r1}). Obviously, the equality in (\ref{H(r1)}) holds if and only if the equality in (\ref{Hr-XX}) holds if and only if $$\tilde J=Z=\sum\limits_{i=1}^{n-1}\lambda^i(t)\tilde e_i(t), \ \ \ \ \  g_T\left(R_T(J), J\right)=\tilde g_{\tilde T}\left(\tilde R_{\tilde T}(Z),  Z\right)=\kappa\sum_{i=1}^{n-1}(\lambda^i)^2.$$  The latter equality implies that $R_{T}(J)=\kappa J$ along $ \gamma$ since $g_T(T, J)=0$. In particular, $R_{\dot\gamma}(X)=R_{\dot\gamma}(J(\ell))=\kappa J(\ell)=\kappa X$. Hence $\mathbf K(\nabla r)=\kappa$ by the arbitrariness of $X$.  On the other hand,  the equality in (\ref{Delta-r1}) holds if and only if
\beq {\rm{tr}}_{\nabla r}(\nabla^2r) -(n-1)\mathfrak{ct}_\kappa(r)=\mathbf S(\nabla r)-(n-1)h \label{rS}\eeq for $0 < r\leq \min\{\mathfrak{i}_{o}, \pi/\sqrt{\kappa}\}$. By the assumptions, the RHS of (\ref{rS}) is non-positive, while (\ref{H(r1)}) shows that the LHS of (\ref{rS}) is nonnegative. Hence the both sides of (\ref{rS}) are identically zero. Thus (\ref{KS}) holds. The converse is obvious. Thus the first claim follows.

 Further, in the polar coordinate system $(r, y)$, we rewrite $d\mathfrak m$ as in the form (\ref{dm-polar}).  Since $r=d_F(o, \cdot)$ is smooth for $0<r< \mathfrak i_o$, by (\ref{loc-Lap}), we have $\Delta r=\frac{\partial}{\partial r}\log\hat\sigma_o(r, y)$ for $0<r<\mathfrak i_o$. Plugging this into (\ref{Delta-r1}) yields
\beq \frac {\partial}{\partial r}\log\hat\sigma_o(r, y)\geq (n-1)\left(\frac{\partial}{\partial r}\log\mathfrak s_\kappa(r)-h\right), \label{sigma-s} \eeq where $0<r\leq \min\{\mathfrak i_o, \pi/\sqrt{\kappa}\}$.  Replacing $r$ by $s$ in (\ref{sigma-s}) and then integrating this on both sides in $s$ from $t$ to $r$ yields
\beqn \hat\sigma_o(t, y)\leq \hat\sigma_o(r, y)\left(\frac{e^{-ht}\mathfrak s_\kappa(t)}{e^{-hr}\mathfrak s_\kappa(r)}\right)^{n-1}.\eeqn
Integrating this inequality on both sides in $t$ from $0$ to $r$ leads to
\beqn \hat{\sigma}_{o}(r, y)\geq (e^{-hr}\mathfrak{s}_{\kappa}(r))^{n-1} \frac{\int_0^r\hat{\sigma}_{o}(t,y)dt}{\int_0^r(e^{-ht}\mathfrak{s}_{\kappa}(t))^{n-1}dt}.\eeqn
Further, by integrating the above inequality over $S_oM$, one obtains
\beq\label{int-SM}
 \int_{S_oM}\hat{\sigma}_{o}(r, y)d\nu_o(y)\geq (e^{-hr}\mathfrak{s}_{\kappa}(r))^{n-1}\frac{\int_{S_oM} d\nu_o(y)\int_0^r\hat{\sigma}_{o}(t, y)dt}{\int_0^r(e^{-ht}\mathfrak{s}_{\kappa}(t))^{n-1}dt}.
\eeq
 Since $m(B^+_r(o))=\int_{S_oM} d\nu_o(y)\int_0^r\hat{\sigma}_{o}(t, y)dt$ and $\mathcal{V}_{o, \kappa, h, n}(r)=\chi(r) \mathfrak{L}_{\mathfrak{m}}(o)$, their derivatives w.r.t. the variable $r$ are given by  \beq {m}'(B^+_r(o))=\int_{S_oM}\hat{\sigma}_{o}(r, y)d\nu_o(y), \ \ \ \ \ \  \ \mathcal{V}'_{o,\kappa, h, n}(r)=(e^{-hr} \mathfrak{s}_\kappa(r))^{n-1} \mathfrak{L}_{\mathfrak{m}}(o).\label{mv-der}\eeq
From this, one obtains
\beqn
\frac{d}{dr}\left(\log f_1(r)\right)&=& \frac{d}{dr}\big(\log{m}(B^+_r(o))-\log\mathcal{V}_{o,\kappa, h, n}(r)\big)\\\nonumber
&=&\frac{{m}'(B^+_r(o))}{{m}(B^+_r(o))}-\frac{\mathcal{V}'_{o,\kappa, h, n}(r)}{\mathcal{V}_{o,\kappa, h, n}(r)}\\\nonumber
&=&\frac{\int_{S_oM}\hat{\sigma}_{o}(r, y)d\nu_o(y)}{\int_{S_oM} d\nu_o(y)\int_0^r\hat{\sigma}_{o}(t,y)dt}-\frac{(e^{-hr}  \mathfrak{s}_\kappa(r))^{n-1}}{\int_0^r (e^{-ht} \mathfrak{s}_\kappa(t))^{n-1}dt}\geq 0, \eeqn where we used $(\ref{int-SM})$ in the last inequality.
Thus $f_1(r)$ is non-decreasing in $r$. By (\ref{sigma-s}), we have
\beq \frac{\partial}{\partial r}\left(\frac{\hat{\sigma}_{o}(r, y)}{(e^{-hr}\mathfrak{s}_{\kappa}(r))^{n-1}}\right)\geq 0,\label{sigma*} \qquad \ 0<r<\min\{\mathfrak{i}_{o}, {\pi}/{\sqrt{\kappa}}\}.  \eeq
Note that $({e^{-hr}\mathfrak{s}_\kappa}(r))^{n-1}\sim r^{n-1}$ as $r\rightarrow0^+$. From this and (\ref{sigma-limit}), we get $\lim\limits_{r\rightarrow 0^+}\frac{\hat{\sigma}_{o}(r, y)}{({e^{-hr}\mathfrak{s}_\kappa}(r))^{n-1}}=e^{-\tau(y)}$ and hence
\beq  \hat{\sigma}_{o}(r, y)\geq e^{-\tau(y)}\big({e^{-hr}\mathfrak{s}_\kappa}(r)\big)^{n-1}. \label{hat-sigma*}\eeq
Integrating this over $B^+_r(o)$ yields
$${m}(B^+_r(o))\geq \int_{S_oM}e^{-\tau(y)} d\nu_o(y) \int_0^r\big(e^{-hr}\mathfrak{s}_{\kappa}(r)\big)^{n-1}dt=\mathcal{V}_{o,\kappa, h, n}(r).$$
i.e., $f_1(r)\geq 1$. If the equality in (\ref{Delta-r1}) holds, then all inequalities in (\ref{sigma-s}) and  after (\ref{sigma-s}) become identities. So $f_1(r)=1$. Conversely, if $f_1(r)=1$, then ${m}(B^+_r(o))=\mathcal{V}_{o,\kappa, h, n}(r)$. Taking the derivatives with respect to $r$ on both sides of this equality and using (\ref{mv-der}), (\ref{tau-integ}), (\ref{hat-sigma*}) yield $\hat{\sigma}_{o}(r, y)=e^{-\tau(y)}(\mathfrak{s}_{\kappa}(r)e^{-hr})^{n-1}$. From this and (\ref{sigma*}), we have
$$e^{-\tau(y)}=\frac{\hat{\sigma}_{o}(r, y)}{(\mathfrak{s}_{\kappa}(r)e^{-hr})^{n-1}}\geq \lim\limits_{r\rightarrow 0^+}\frac{\hat{\sigma}_{o}(r, y)}{(\mathfrak{s}_{\kappa}(r)e^{-hr})^{n-1}}=e^{-\tau(y)},$$ which means that
$$\frac{\hat{\sigma}_{o}(r, y)}{(\mathfrak{s}_{\kappa}(r)e^{-hr})^{n-1}}=\lim\limits_{r\rightarrow 0^+}\frac{\hat{\sigma}_{o}(r, y)}{(\mathfrak{s}_{\kappa}(r)e^{-hr})^{n-1}}=e^{-\tau(y)}.$$ By the arbitrariness of $r$, the equality in (\ref{sigma*}) holds. From the proof before (\ref{sigma*}), we can deduce that the equality in (\ref{sigma-s}) holds, i.e., (\ref{Delta-r1}) becomes equality. Thus $f_1(r)=1$ if and only if the equality in (\ref{Delta-r1}) holds.  Consequently, $f_1(r)=1$ if and only if (\ref{KS}) holds.\end{proof}

\begin{rem} {\rm In \cite{WX}, the authors gave the Laplacian comparison (\ref{Delta-r1}) and the volume comparison without characterizations of the equalities under the assumptions that $\mathbf K\leq \kappa$ and $\mathbf S\leq (n-1)h$. Actually,  Theorem \ref{thm21} gives (\ref{Delta-r1}) under more weaker curvature conditions. Based on this, we prove the volume comparison theorem and characterize the sharpness for these two comparisons. The volume comparison result in Theorem \ref{thm21} is different from that in \cite{WX}.}
\end{rem}

Based on Theorem \ref{thm21}, we obtain the following corollary.

\begin{cor}\label{cor31} Let $(M, F, \mathfrak{m})$ be an $n$-dimensional forward complete Finsler measure space satisfying
\beqn \mathbf K\leq \kappa, \ \ \ \ \mathbf S\leq (n-1)h,  \eeqn for some $\kappa, h\in \mathbb R$ satisfying either $h=\kappa=0$, or $h>0$ if $\kappa\geq 0$  and $h>\sqrt{-\kappa}>0$ if $\kappa<0$. If there is some $o\in M$ such that
 $$ \mathfrak{L}_{\mathfrak{m}}(o)=\inf_{x\in M}  \mathfrak{L}_{\mathfrak{m}}(x), \ \ \ \ m(B_r^+(o))=\chi(r) \mathfrak{L}_{\mathfrak{m}}(o),$$ where $r=d_F(o, \cdot)$, then $\mathfrak m=\mathfrak m_{BH}$, $\mathbf K\equiv \kappa$ and $\mathbf S\equiv (n-1)h$. In this case, $m(B_{r}^+(x)=\mathcal{V}_{o,\kappa, h, n}(r)=\chi(r) \mathfrak{L}_{\mathfrak{m}}(o)$ only depending on the radius $r$ for any $x\in M$. \end{cor}
\begin{proof} Since Corollary \ref{cor31} was obtained in \cite{HKZ} in the case when $h=\kappa=0$, it suffices to consider the cases when $h>0$ if $\kappa\geq 0$,  and $h>\sqrt{-\kappa}>0$ if $\kappa<0$.  In these cases, we first remark that the integral $\int_0^{+\infty}(e^{-ht}\mathfrak s_{\kappa}(t))^{n-1}dt$ converges, i.e., $\lim_{r\rightarrow +\infty}\chi(r)<+\infty$, by the assumptions on $\kappa$ and $h$.
Fixed $x\in M$ but arbitrarily,  let $\tilde r:=r_x=d_F(x, \cdot)$. Then  $m(B_{\tilde r}^+(x))\geq \mathcal{V}_{x,\kappa, h, n}(\tilde r)=\chi(\tilde r) \mathfrak{L}_{\mathfrak{m}}(x)$ and $f_1(\tilde r)=\frac{m(B^+_{\tilde r}(x))}{\chi(\tilde r)}$ is non-decreasing by Theorem \ref{thm21}. Since $B_{\tilde r}^+(x)\subset B^+_{\tilde r+r(x)}(o)$, we have
\beq  \mathfrak{L}_{\mathfrak{m}}(x)&\leq &\frac{m(B_{\tilde r}^+(x))}{\chi(\tilde r)} \leq \limsup_{\tilde r\rightarrow +\infty}\frac{m(B_{\tilde r}^+(x))}{\chi(\tilde r)}\nonumber \\
& \leq & \limsup_{\tilde r\rightarrow +\infty}\frac{m(B^+_{\tilde r+r(x)}(o))}{\chi(\tilde r)}\nonumber \\
&=& \limsup_{\tilde r\rightarrow +\infty}\left(\frac{ m(B^+_{\tilde r+r(x)}(o))}{\chi(\tilde r+r(x))}\cdot \frac{\chi(\tilde r+r(x))}{\chi(\tilde r)}\right)\nonumber \\
&=&  \mathfrak{L}_{\mathfrak{m}}(o)\cdot \limsup_{\tilde r\rightarrow +\infty}\left(\frac{\chi(\tilde r)+\int_{\tilde r}^{\tilde r+r(x)}(e^{-ht}\mathfrak s_{\kappa}(t))^{n-1}dt}{\chi(\tilde r)}\right), \nonumber \\
&=& \mathfrak{L}_{\mathfrak{m}}(o)\leq  \mathfrak{L}_{\mathfrak{m}}(x),\label{L(x)} \eeq
where we used
$\lim_{\tilde r\rightarrow +\infty}\int_{\tilde r}^{\tilde r+r(x)}(e^{-ht}\mathfrak s_{\kappa}(t))^{n-1}dt=0$. In fact, by the mean value theorem and definition of $\mathfrak s_{\kappa}(t)$ with the assumptions on $\kappa, h$, there is some $\bar t\in (\tilde r, \tilde r+r(x))$ such that
\beqn \lim_{\tilde r\rightarrow +\infty}\int_{\tilde r}^{\tilde r+r(x)}(e^{-ht}\mathfrak s_{\kappa}(t))^{n-1}dt=\lim_{\bar t\rightarrow +\infty}r(x)(e^{-h\bar t}\mathfrak s_{\kappa}(\bar t))^{n-1}=0.\eeqn
Then (\ref{L(x)}) implies that $ \mathfrak{L}_{\mathfrak{m}}(x)= \mathfrak{L}_{\mathfrak{m}}(o)$ and $m(B_{\tilde r}^+(x))=\chi(\tilde r) \mathfrak{L}_{\mathfrak{m}}(x)=\mathcal{V}_{x,\kappa, h, n}(\tilde r)$. By Proposition \ref{prop21} and Theorem \ref{thm21}, we have $\mathfrak m=\mathfrak m_{BH}$, $\mathbf {K}(\nabla \tilde r)=\kappa$ and $\mathbf S( \nabla \tilde r)=(n-1)h$. In this case, $m(B_{\tilde r}^+(x))=m(B_{\tilde r}^+(o))=\chi(\tilde r) \mathfrak{L}_{\mathfrak{m}}(o)$ only depending on the radius $\tilde r$ for any $x\in M$. Assume that $\gamma: [0, \tilde r]\rightarrow M$ be a normal minimal geodesic with $\gamma(0)=x$ and $\dot\gamma(0)=y\in T_xM\setminus\{0\}$. Then $\dot\gamma(t)=\nabla \tilde r$. Thus, $\mathbf K(\dot\gamma(t))=\kappa$ and $\mathbf S(\dot \gamma(t))=(n-1)h$. Letting $t\rightarrow 0$ yields $\mathbf K(x, y)=\kappa$ and $\mathbf S(x, y)=(n-1)h$. By the arbitrariness of $(x, y)$,  $\mathbf K\equiv \kappa$ and $\mathbf S\equiv (n-1)h$.
\end{proof}

\begin{thm} \label{thm22}
Let $(M, F, \mathfrak{m})$ be an $n$-dimensional forward complete Finsler manifold and $r=d_F(o, \cdot)$ be the distance function from some $o\in M$.  If  $\mathbf{Ric}(\nabla r)\geq(n-1)\kappa$ and ${\mathbf S}(\nabla r)\geq(n-1)h$ for some $\kappa, h\in\mathbb{R}$, then
\beq
	\Delta r\leq(n-1)\big(\mathfrak{ct}_\kappa(r)-h\big), \ \ \ \ \ \ 0 < r\leq \min\{\mathfrak{i}_{o}, \pi/\sqrt{\kappa}\}, \label{Delta-r2}\eeq in which the equality holds if and only if (\ref{KS}) holds.
Further, the function
\beqn  f_2(r):=\frac{m(B^+_{r}(o))}{\nu_{o, \kappa, h, n}(r)} \eeqn is non-increasing in $r$, where $\mathcal{V}_{o,\kappa,h,n}(r)=\chi(r) \mathfrak{L}_{\mathfrak{m}}(o)$ as in Theorem \ref{thm21}.
In particular, $f_2(r)\leq 1$ and hence $
m(B^+_{o}(r))\leq \mathcal{V}_{o,\kappa, h, n}(r)$. $f_2(r)=1$ if and only if (\ref{KS}) holds. In this case, (\ref{hat-sigma}) holds.
 \end{thm}

\begin{proof}  The proof of (\ref{Delta-r2}) is similar to that of Theorem 5.3 in \cite{WX}. To argue the case of equality, we give the sketch of proof.   For any $x\in M$, let $\gamma: [0, \ell]\rightarrow M$ be a normal minimal geodesic from $o=\gamma(0)$ to $x=\gamma(\ell)$ such that $x$ is not a cut point of $o$. Then $r(\gamma(t))=d_F(o, \gamma(t))$ is smooth and $\dot \gamma(t) =\nabla r(\gamma(t))$ with $F(\dot\gamma)=1$ for $0<r\leq \mathfrak i_{o}$.  Let $e_1, \ldots, e_n=\dot\gamma(t)$ be the parallel $g_T$-orthonormal frame fields along $\gamma$.  For $1\leq i\leq n-1$, let $J_i(t)$ be the unique Jacobi field along $\gamma$ with $J_i(0)=0, J_i(\ell)=e_i(\ell)$ and $W_i(t):=\frac {\mathfrak s_\kappa(t)}{\mathfrak s_\kappa(\ell)}e_i(t)$, where $\mathfrak s_\kappa(t)$ is defined as before. Obviously, we have $W_i(0)=J_i(0)=0$ and $W_i(\ell)=J_i(\ell)$. From Propositions 6.3.1--6.3.2 in \cite{Xia} and the basic index lemma (\cite{BCS}, \cite{Sh}, \cite{Xia}), we have
\beq {\rm{tr}}_{\nabla r}(\nabla^2r)|_x&=&\sum\limits_{i=1}^n\nabla^2r(e_i(\ell), e_i(\ell))=\sum\limits_{i=1}^{n-1}I_{\gamma}(J_i, J_i)\leq \sum\limits_{i=1}^{n-1}I_{\gamma}(W_i, W_i)\nonumber \\
&=&\frac 1{\mathfrak s_{\kappa}(\ell)^2}\int_0^{\ell}\left\{(n-1)\mathfrak s_\kappa'(t)^2-{\mathbf{Ric}}(\nabla r)\mathfrak s_\kappa(t)^2\right\}dt\label{H(r)} \\
&\leq & \frac 1{\mathfrak s_{\kappa}(\ell)^2}\int_0^{\ell}\left\{(n-1)\mathfrak s_\kappa'(t)^2-(n-1)\kappa \mathfrak s_\kappa(t)^2\right\}dt=(n-1)\mathfrak{ct}_\kappa(r(x)).\nonumber \eeq Thus, (\ref{Delta-r2}) follows from (\ref{H(r)}) and (\ref{Delta u}). Obviously, the equality in (\ref{H(r)}) holds if and only if $J_i(t)=W_i(t)=\frac {\mathfrak s_\kappa(t)}{\mathfrak s_\kappa(\ell)}e_i(t)$ and $\mathbf{Ric}(\nabla r)=(n-1)\kappa$. From the Jacobi equation, we obtain
$$ R_{\dot\gamma}(e_i)=-R_{\dot\gamma}(\dot\gamma, e_i)\dot\gamma=\kappa e_i, $$ which means that $\mathbf K(\nabla r, \cdot)=\mathbf K(\dot\gamma, \cdot)=\kappa$ by (\ref{flag-curv}). Note that the equality in (\ref{Delta-r2}) holds if and only if (\ref{rS}) holds for $0 < r\leq \min\{\mathfrak{i}_{o}, \pi/\sqrt{\kappa}\}$.  By the assumption, the RHS of (\ref{rS}) is nonnegative. On the other hand,  (\ref{H(r)}) implies that the LHS of (\ref{rS}) is nonpositive. Hence the both sides of (\ref{rS}) are identically zero. Thus (\ref{KS}) holds. The converse is obvious. The rest proof is similar to that of Theorem \ref{thm21}. This finishes the proof.
\end{proof}

Based on Theorem \ref{thm22}, we have the following result.

\begin{cor}\label{cor32} Let $(M, F, \mathfrak{m})$ be an $n$-dimensional forward complete Finsler measure space satisfying
\beqn \mathbf {Ric}\geq (n-1)\kappa, \ \ \ \ \mathbf S\geq (n-1)h,  \eeqn for some $\kappa, h\in \mathbb R$ satisfying either $h=\kappa=0$, or $h>0$ if $\kappa\geq 0$ and $h>\sqrt{-\kappa}>0$ if $\kappa<0$. If there is some $o\in M$ such that
 $$ \mathfrak{L}_{\mathfrak{m}}(o)=\sup_{x\in M}  \mathfrak{L}_{\mathfrak{m}}(x), \ \ \ \ m(B_r^+(o))=\chi(r) \mathfrak{L}_{\mathfrak{m}}(o),$$ where $r=d_F(o, \cdot)$,  then $\mathfrak m=\mathfrak m_{BH}$, $\mathbf K\equiv \kappa$ and $\mathbf S\equiv (n-1)h$. In this case, $m(B_{ r}^+(x))=\mathcal{V}_{o,\kappa, h, n}( r)=\chi( r) \mathfrak{L}_{\mathfrak{m}}(o)$ only depending on the radius $r$ for any $x\in M$. \end{cor}
\begin{proof} The proof is similar to that of Corollary \ref{cor31}. Since Corollary \ref{cor32} was obtained in \cite{HKZ} in the case when $h=\kappa=0$, it suffices to consider the cases when $h>0$ if $\kappa\geq 0$, and $h>\sqrt{-\kappa}>0$ if $\kappa<0$. In theses cases, we first remark that $\lim_{r\rightarrow +\infty}\chi(r)<+\infty$ as in the proof of Corollary \ref{cor31}.
Fixed $x\in M$ but arbitrarily, let $\tilde r:=r_x=d_F(x, \cdot)$. Then $m(B_{\tilde r}^+(x))\leq \chi(\tilde r) \mathfrak{L}_{\mathfrak{m}}(x)$ and $f_2(\tilde r)=\frac{m(B^+_{\tilde r}(x))}{\chi(\tilde r)}$ is non-increasing by Theorem \ref{thm22}. Since $B_{\tilde r}^+(x)\supset B^+_{\tilde r-d_F(x, o)}(o)$ for sufficiently large $\tilde r>0$, we have
\beq  \mathfrak{L}_{\mathfrak{m}}(x)&\geq &\frac{m(B_{\tilde r}^+(x))}{\chi(\tilde r)} \geq \liminf_{\tilde r\rightarrow +\infty}\frac{m(B_{\tilde r}^+(x))}{\chi(\tilde r)}\nonumber \\
& \geq & \liminf_{\tilde r\rightarrow +\infty}\frac{m(B^+_{\tilde r-d_F(x, o)})}{\chi(\tilde r)}\nonumber \\
&=& \liminf_{\tilde r\rightarrow +\infty}\left(\frac{ m(B^+_{\tilde r-d_F(x, o)}(o))}{\chi(\tilde r-d_F(x, o))}\cdot \frac{\chi(\tilde r-d_F(x, o))}{\chi(\tilde r)}\right)\nonumber \\
&=& \mathfrak{L}_{\mathfrak{m}}(o)\cdot \liminf_{\tilde r\rightarrow +\infty}\left(\frac{\chi(\tilde r)+\int_{\tilde r}^{\tilde r-d_F(x, o)}(e^{-ht}\mathfrak s_{\kappa}(t))^{n-1}dt}{\chi(\tilde r)}\right), \nonumber \\
&=& \mathfrak{L}_{\mathfrak{m}}(o)\geq  \mathfrak{L}_{\mathfrak{m}}(x),\label{L(x)*} \eeq
where we used  $\lim_{\tilde r\rightarrow +\infty}\int_{\tilde r}^{\tilde r-d_F(x, o)}(e^{-ht}\mathfrak s_{\kappa}(t))^{n-1}dt=0$ as in the proof of Corollary \ref{cor31}.
(\ref{L(x)*}) implies that $ \mathfrak{L}_{\mathfrak{m}}(x)= \mathfrak{L}_{\mathfrak{m}}(o)$ and $m(B_{\tilde r}^+(x))=\chi(\tilde r) \mathfrak{L}_m(x)=\mathcal{V}_{x,\kappa, h, n}(\tilde r)$ for any $x\in M$. By Proposition \ref{prop21} and Theorem \ref{thm22}, we have $\mathfrak m=\mathfrak m_{BH}$, $\mathbf {K}(\nabla \tilde r)=\kappa$ and $\mathbf S( \nabla \tilde r)=(n-1)h$. In this case, $m(B_{ r}^+(x))=\chi(\tilde r) \mathfrak{L}_{\mathfrak{m}}(o)$ only depending on $\tilde r$ for any $x\in M$.  The rest follows from  the proof of Corollary \ref{cor31}.
\end{proof}

The following equivalence follows directly from  Theorem \ref{thm21} or Theorem \ref{thm22}.
\begin{cor}\label{cor33} Under the same assumptions as in Theorem \ref{thm21} or Theorem \ref{thm22}, the following statements are equivalent.

{\rm {(1)}} $\Delta r =(n-1)\big(\mathfrak{ct}_\kappa(r)-h\big)$ for $0 < r\leq \min\{\mathfrak{i}_{o}, \pi/\sqrt{\kappa}\}$.

{\rm{(2)}} $m(B_{r}^+(o))=\mathcal{V}_{o,\kappa, h, n}(r)$.

{\rm{(3)}} $\mathbf {K}(\nabla r)=\kappa$ and $\mathbf S(\nabla r)=(n-1)h$.
\end{cor}

\section{Generic functional inequalities}
In this section, we establish two generic functional inequalities on a Finsler measure space $(M, F, \mathfrak{m})$ by the convexity of $F$ inspired from the arguments in  Riemannian setting (\cite{KKPZ}). Further, we give equivalent characterizations for their sharpness. For simplicity, we state our results for $u\in C_0^\infty(M)$. Its extension to a more general class of functions, e.g. $\mathcal W_0^{1, p}(M)$ or $W_0^{1, p}(M)$, is standard.

\begin{prop}\label{prop31}
Let $(M, F, \mathfrak{m})$ be an $n$-dimensional forward complete Finsler measure space. For any domain $\Omega$ in $M$ and a positive function $\rho\in \mathcal W_{loc}^{1, p}(\Omega)$ (resp. $W_0^{1, p}(\Omega)$)  with $F^{\ast}(d\rho)=1$ a.e on $\Omega\subset M$, let $G: (0, \sup_{\Omega}\rho) \rightarrow \mathbb{R}$ and $H: \mathbb{R} \rightarrow \mathbb{R}$ be $C^1$ functions  such that $G(\rho)\in L^{p'}_{loc}(\Omega)$, $G'(\rho)\in L^{1}_{loc}(\Omega)$ and $H(0)=H'(0)=0$, where $p'=\frac p{p-1}$ and $p>1$. Then the following inequalities hold.

(1) For any $u\in C^{\infty}_0(\Omega)$, we have
\beq
\int_{\Omega}\max\left\{F^{\ast p}(\pm du)\right\}d\mathfrak{m}\geq p\int_{\Omega}\left\{G'(\rho)+G(\rho)\Delta \rho\right\}H(u)d\mathfrak{m}-(p-1)\int_{\Omega}|G(\rho)H'(u)|^{p'}d\mathfrak{m}.\label{3.1}\eeq

(2) For any $u\in C^{\infty}_0(\Omega)$, we have
\beq
	\int_{\Omega}\max\left\{F^{\ast p}(\pm du)\right\}d\mathfrak{m} \geq \frac{\left|\int_{\Omega}\left\{G'(\rho)+G(\rho)\Delta \rho\right\}H(u)d\mathfrak{m}\right|^p}{\big(\int_{\Omega}|G(\rho)H'(u)|^{p'}d\mathfrak{m}\big)^{p-1}}.\label{3.2}
\eeq
provided that there exists a neighborhood $I\subset \mathbb{R}$ of zero such that $H'(s)\neq 0$ for any $s\in I\backslash\{0\}$ and $G(t)\neq 0$ for all $t\in (0, \sup_{\Omega}\rho)$. Further, the equality in (\ref{3.1}) or (\ref{3.2}) holds if and only if
\beq du=-{\rm{sgn}}(G(\rho)H'(u))|G(\rho)H'(u)|^{\frac 1{p-1}}d\rho, \ \ \  \max\left\{F^{\ast}(\pm d\rho)\right\}= F^{\ast}(d\rho) \ \ \ {\rm{on}}\ \Omega.
\label{u-rho} \eeq
 \end{prop}

\begin{proof} (1) The convexity of the map $\xi\mapsto F^{{\ast}p}(\xi)$ implies
\beqn
F^{{\ast}p}(\xi) \geq F^{\ast p}(\eta)+\big\langle\xi-\eta, \ \frac{\partial}{\partial \eta} F^{\ast p}(\eta)\big\rangle  \label{F*-xi-eta}
\eeqn for any $\xi, \eta\in \Gamma(T^{\ast}\Omega)$ (i.e., the set of smooth sections on $T^{\ast}\Omega$) and the equality holds if and only if $\xi=\eta$. By (\ref{FF*})--(\ref{F*-xi}), one obtains
\beq
F^{\ast p}(\xi) \geq p F^{\ast (p-2)}(\eta)\langle\xi, \mathcal L^{\ast}(\eta)\rangle-(p-1) F^{\ast p}(\eta).\label{F*-xi-eta}
\eeq
Let
\beqn
\xi: =-\operatorname{sgn}\big(G(\rho)H'(u)\big)du,  \qquad  \eta: = |G(\rho)H'(u)|^{\frac{1}{p-1}}d\rho \eeqn for any $u\in C_0^\infty(\Omega)$.
Then
\beqn & p F^{\ast (p-2)}(\eta)=p |G(\rho)H'(u)|^{\frac{p-2}{p-1}}F^{\ast (p-2)}(d\rho),\\
& (p-1) F^{\ast p}(\eta)=(p-1) |G(\rho)H'(u)|^{p^{\prime}}F^{\ast p}(d \rho).\eeqn
Since  $\mathcal L^{\ast}(d\rho)=\nabla \rho$, we get
$$\langle\xi, \mathcal L^{\ast}(\eta)\rangle=-|G(\rho)H'(u)|^{\frac{2-p}{p-1}}G(\rho)H'(u) du(\nabla\rho).$$
% \beqn pF^{{\ast}(p-2)}(\eta)\langle\xi,\ell^{\ast}(\eta)\rangle=-pG(\rho)F^{{\ast}(p-2)}(d\rho) dH(u)(\nabla\rho).\eeqn
Plugging these in (\ref{F*-xi-eta}) yields
$$ F^{\ast p}\big(-{\rm{sgn}}\big(G(\rho)H'(u)\big)du\big)\geq-pG(\rho)F^{\ast (p-2)}(d\rho)d(H(u))(\nabla\rho)-(p-1)|G(\rho)H'(u)|^{p'}F^{\ast p}(d\rho).$$
Note that $F^{\ast}(d\rho)=1$ a.e. on $\Omega$ by the assumption. Integrating the above inequality over $\Omega$ yields
\beqn
\int_{\Omega}F^{\ast p}\big(-\operatorname{sgn}\big(G(\rho)H'(u)\big)du\big)d\mathfrak{m}&\geq&-p\int_{\Omega}G(\rho)d(H(u))(\nabla\rho)d\mathfrak{m}\\
& &-(p-1)\int_{\Omega}|G(\rho)H'(u)|^{p'}d\mathfrak{m}\nonumber \\
&: =& pI_H-(p-1)J_H. \eeqn
Since  $ \max\left\{F^{\ast}(\pm du)\right\}\geq F^{\ast}\big(-\operatorname{sgn}\big(G(\rho)H'(u)\big)du\big)$, we have
\beq\label{3.8}
\int_{\Omega}\max\{F^{{\ast}p}(\pm du)\}d\mathfrak{m}\geq p I_H -(p-1)J_H.
\eeq
Since $u\in C_0^\infty(\Omega)$ and $H\in C^1(\mathbb R)$ with $H(0)=0$,  it follows from an integration by parts and (\ref{fV}) that
\beqn
I_H=-\int_{\Omega} d\big(H(u)\big) \big(G(\rho)\nabla\rho\big) d\mathfrak{m}=\int_{\Omega} \big\{G'(\rho)+G(\rho)\Delta\rho\big\}H(u)d\mathfrak{m}.
\eeqn
Thus $(\ref{3.1})$ follows. The equality in (\ref{3.1}) holds if and only if all inequalities in previous proof become equalities, i.e.,
\beq \xi=\eta, \ \ \ \ \  \max\left\{F^{\ast}(\pm du)\right\}=F^{\ast}\big(-\operatorname{sgn}\big(G(\rho)H'(u)\big)du\big).\label{xi-eta-rho}\eeq  (\ref{xi-eta-rho})$_1$ is exactly (\ref{u-rho})$_1$ by definitions of $\xi$ and $\eta$. Plugging (\ref{u-rho})$_1$ in (\ref{xi-eta-rho})$_2$ yields (\ref{u-rho})$_2$. Conversely, if (\ref{u-rho}) holds, then (\ref{xi-eta-rho}) holds obviously. Thus  the equality in (\ref{3.1}) holds if and only if (\ref{u-rho}) holds.

(2) Observe that $(\ref{3.8})$ is valid for $cH$ instead of $H$ for every $c\in \mathbb{R}\backslash\{0\}$. Consequently,
\beqn
\int_{\Omega}\max\{F^{{\ast}p}(\pm Du)\}d\mathfrak{m}\geq p I_{cH} -(p-1)J_{cH}=pcI_H-(p-1)|c|^{p'}J_H.
\eeqn
Since $H'(s)\neq 0$ on $I\backslash\{0\}$ and $G(t)\neq 0$ on $(0, \sup_{\Omega}\rho)$, we have $J_H>0$. We can maximize the right hand side of the latter expression with respect to $c$. To keep the sign of $c$  consistent with that of $I_H$, let $f(t)=p|I_H| t-(p-1)J_H|t|^{p'}$,  where $t=|c|=\text{sgn}(I_H)c>0$. Consequently $f'(t)=p|I_H|-p J_H t^{p'-1}$ and $f''(t)=-p(p'-1)J_H t^{p'-2}<0$. The maximum of $f$ is achieved at $t=J_H^{1-p}|I_H|^{p-1}$, i.e., $c=J_H^{1-p}|I_H|^{p-2}I_H$. Hence
$$ \int_{\Omega}\max\{F^{{\ast}p}(\pm du)\}d\mathfrak{m}\geq \frac{|I_H|^{p}}{J_H^{p-1}}.$$ Thus, (\ref{3.2}) follows.
 Since the equality of  (\ref{3.2}) follows from that of (\ref{3.1}), the equality in (\ref{3.2}) holds if and only if (\ref{u-rho}) holds.
This finishes the proof.\end{proof}

To extend the integrals over $\Omega$ in (\ref{3.1})--(\ref{3.2}) to those over $M$, let us introduce the capacity for certain small subset of $M$ connected with the Sobolev space $\mathcal W^{1, p}(M)$ (resp. $W^{1, p}(M)$) (cf. \cite{EG}, \cite{HKM}). Since the arguments are similar, we only consider the Sobolv space $W^{1, p}(M)$ instead of $\mathcal W^{1, p}(M)$. In fact, the following definition of $p$-capacity also works and the corresponding results are also true for $\mathcal W^{1, p}(M)$.

 For any $p>1$ and a subset $\mathcal A\subset M$, define the {\it $p$-capacity} of $\mathcal A$ by ${\rm {Cap}}_p(\mathcal A)=\inf_{u\in \mathcal C(\mathcal A)}\|u\|_{1, p},$ where $\mathcal C(\mathcal A)=\left\{u\in W^{1, p}(M)| \mathcal A\subset \{u\geq 1\}^o\right\}$, here $\{u\geq 1\}^o$ means that the interior of the set $\{x\in M|u(x)\geq 1\}$ (cf. \cite{EG}).  It has an equivalent definition as follows.
\begin{prop} \label{prop32} For any set $\mathcal A\subset M$,
$${\rm {Cap}}_p(\mathcal A)=\inf_{u\in \mathcal C'(\mathcal A)}\|u\|_{1, p},$$
where $\mathcal C'(\mathcal A)=\left\{u\in W^{1, p}(M)|0\leq u\leq 1\  {\rm {and}}\ \mathcal A\subset \{u=1\}^o\right\}$. In particular, {\rm {Cap}}$_p(r_o^{-1}(0))=0$ if $1<p<n$, where $r_o=d_F(o, \cdot)$ is the distance function from some $o\in M$. \end{prop}
\begin{proof}  Since $\mathcal C'(\mathcal A)\subset \mathcal C(\mathcal A)$, we have ${\rm {Cap}}_p(\mathcal A)\leq \inf_{u\in \mathcal C'(\mathcal A)}\|u\|_{1, p}$.
On the other hand, for any $\eta>0$, there exists a function $\bar u\in \mathcal C(\mathcal A)$ such that $\|\bar u\|_{1, p}<{\rm {Cap}}_p(\mathcal A)+\eta$ by definition. Set $v: =\max\{0, \min\{\bar u, 1\}\}$. Then $|v|\leq |\bar u|$ and $0\leq v\leq 1$. Moreover, by Lemma 7.6 in \cite{GT} (also see the proof of Theorem 1.1 in \cite{ZX}), we have $dv=0$ on $\{\bar u\geq 1 \ {\rm{or}}\ \bar u\leq 0\}$ and $dv=d\bar u$ on $\{0\leq \bar u\leq 1\}$ in the weak sense. Thus, $v\in W^{1, p}(M)$. Since $\bar u\in \mathcal C(\mathcal A)$, we have $\mathcal A\subset \{\bar u\geq 1\}^o$, which implies that $\mathcal A\subset \{v=1\}^o$. Consequently, $v\in \mathcal C'(\mathcal A)$ and
$${\rm {Cap}}_p(\mathcal A)\leq \inf_{u\in \mathcal C'(\mathcal A)}\|u\|_{1, p}\leq \|v\|_{1, p}\leq  \|\bar u\|_{1, p}<{\rm {Cap}}_p(\mathcal A)+\eta, $$ which implies that ${\rm {Cap}}_p(\mathcal A)=\inf_{u\in \mathcal C'(\mathcal A)}\|u\|_{1, p}$ by the arbitrariness of $\eta$.

In particular, let $\mathcal A=r_o^{-1}(0)$.  Then $\mathcal A=\{0\}$. For any $0<\epsilon<\delta$, consider a $C^\infty$ function defined by
\beqn u(x) =\left\{\begin{array}{lll} 1 & {\rm{on}}\ B_{\epsilon}, \\
\frac{\delta-r_o(x)}{\delta-\epsilon},  & {\rm{on}}\ B_{\delta}\backslash B_{\epsilon},\\
0, & {\rm{on}}\ M\backslash B_{\delta},\end{array}\right.\eeqn where $B_{\bullet}:=B_{\bullet}^+(o)$.
Then $0\leq u\leq 1$ and $\mathcal A \subset  B_{\epsilon}=\{u=1\}^o$. Next we prove that there is a $\delta>0$ such that $\|u\|_{1, p}$ for $u$ defined as above is finite and hence ${\rm{Cap}}_p(\mathcal A)=0$.

 In fact, for a small $0<\delta<\mathfrak i_o$, let $(r, y)$ be the polar coordinate system around $o$, where $r=r_o(\cdot )<\delta$ and $y\in S_oM$.  According to $(\ref{sigma-limit})$,  there exists an $\varepsilon=\varepsilon(\delta)>0$ such that $\lim_{\delta\rightarrow0^+}\varepsilon(\delta)=0$ and
\beq (1-\varepsilon(\delta))e^{-\tau(y)}r^{n-1} \leq \hat{\sigma}_o(r,y) \leq (1+\varepsilon(\delta))e^{-\tau(y)}r^{n-1}\label{hat-sigma-est}\eeq
for any $(r, y)\in B_{\delta}\backslash\{o\}$. From this, we have
\beqn \frac 1n (1-\varepsilon(\delta))\mathfrak{L}_{\mathfrak{m}}(o)\delta^n\leq m(B_{\delta})=\int_{B_{\delta}}d\mathfrak m=\int_0^\delta\int_{S_o(M)}\hat\sigma_o(r, y)dr d\nu_0(y)\leq \frac 1n (1+\varepsilon(\delta))\mathfrak{L}_{\mathfrak{m}}(o)\delta^n.\eeqn Moreover, $F^{\ast}(-du)=\frac{1}{\delta-\epsilon}$ and $F^{\ast}(du)= \frac{1}{\delta-\epsilon}F^{\ast}(-dr)\leq \frac{\Lambda_\delta}{\delta-\epsilon}$ on $B_\delta\backslash B_\epsilon$, and $F^{\ast}(du)=F^{\ast}(-du)=0$ elsewhere,  where $\Lambda_\delta$ is the reversibility of $F$ on $B_\delta$, which is finite due to the compactness of $\overline{B_\delta}$. Thus, one obtains
\beqn  \|u\|_{1, p} &=&\|u\|_{L^p(B_\delta)} +\|F^{\ast}(du)\|_{L^p(B_\delta)}+\|F^{\ast}(-du)\|_{L^p(B_\delta)}\nonumber \\
&\leq & \left(m(B_\delta)\right)^{1/p}+\frac {1+\Lambda_\delta}{\delta-\epsilon}\left(m(B_\delta)-m(B_\epsilon)\right)^{1/p}\nonumber \\
&\leq& \left(\frac 1n(1+\varepsilon(\delta)) \mathfrak{L}_{\mathfrak{m}}(o)\right)^{1/p}\left\{\delta^{n/p}+\frac{1+\Lambda_\delta}{\delta-\varepsilon}\left(\delta^n-\frac{1-\varepsilon(\delta)}{1+\varepsilon(\delta)}\epsilon^n\right)^{1/p}\right\}.\eeqn
Letting $\epsilon=\frac 12 \delta$ in the above inequality yields
\beq \|u\|_{1, p}\leq \left(\frac 1n(1+\varepsilon(\delta)) \mathfrak{L}_{\mathfrak{m}}(o)\right)^{1/p}\left\{\delta^{n/p}+2(1+\Lambda_\delta)\left(1-\frac {1-\varepsilon(\delta)}{2^n(1+\varepsilon(\delta))}\right)^{1/p}\delta^{\frac{n-p}p}\right\}. \label{u-sob}\eeq
Note that the sequence $\{\Lambda_\delta\}$ is increasing in $\delta$ and $\Lambda_\delta\geq 1$ by definition. Thus the limit $\lim_{\delta\rightarrow 0+}\Lambda_\delta$ exists and is greater than or equal to $1$. Since $\lim_{\delta\rightarrow 0^+}\varepsilon(\delta)=0$ and $1<p<n$, the limit of RHS in (\ref{u-sob}) goes to zero as $\delta\rightarrow 0$. This means that $\lim_{\delta\rightarrow 0^+}\|u\|_{1, p}=0$, i.e.,  for any small $0<\eta<1$, there is a small $\delta>0$ enough such that $\|u\|_{1, p}<\eta<1$. For such a $\delta$, $u\in \mathcal C'(\mathcal A)$. Hence Cap$_p(\mathcal A)=\inf_{u\in \mathcal C'(\mathcal A)}\|u\|_{1, p}=0$ by the arbitrariness of $\eta$.
\end{proof}

\begin{prop}\label{prop33} Let $(M, F, \mathfrak m)$ be an $n$-dimensional Finsler measure space and $\mathcal A\subset M$ be a compact set of zero $p$-capacity. If $1<p<n$, then $W_0^{1, p}(M)\subset W_0^{1, p}(M\backslash \mathcal A)$, that is,  every function in  $W_0^{1, p}(M)$ can be approximated by functions in $C_0^\infty(M\backslash \mathcal A)$ with respect to $\|\cdot\|_{1, p}$. \end{prop}
\begin{proof} For any $\varphi\in C_0^\infty(M)$, it suffices to prove that $\varphi\in W_0^{1, p}(M\backslash\mathcal A)$ by definition. Since Cap$_p(\mathcal A)=0$, there is a sequence $\{u_j\}_{j\geq 1}$ such that $u_j\in C^\infty(M)$, $0\leq u_j\leq 1$ and $\mathcal A\subset \{u_j=1\}^o$ for any $j\geq 1$, and $u_j\rightarrow 0$ in $W^{1, p}(M)$. For every $j\geq 1$, the function $\varphi_j:=(1-u_j)\varphi$ belongs to $C_0^\infty(M\backslash\mathcal A)\subset W_0^{1, p}(M\backslash \mathcal A)$. We shall prove that $\|\varphi_j-\varphi\|_{1, p}\rightarrow 0$ as $j\rightarrow +\infty$, which means that $\varphi\in W_0^{1, p}(M\backslash\mathcal A)$. In fact, since $\|u_j\|_{1, p}\rightarrow 0$ as $j\rightarrow \infty$, we have $\|u_j\|_{L^p(M)}\rightarrow 0$ and $\|F^{\ast}(\pm du_j)\|_{L^p(M)}\rightarrow 0$ as  $j\rightarrow \infty$. Consequently,
 \beq \|\varphi_j-\varphi\|_{L^p(M\backslash \mathcal A)}\leq \| u_j\varphi\|_{L^p(M)}\leq \|\varphi\|_{L^\infty(M)}\|u_j\|_{L^p(M)}\rightarrow 0 \ \ ({\rm{as}} \ j\rightarrow \infty), \label{Lp-1}\eeq
and
 \beqn \|F^{\ast}(d\varphi_j-d\varphi)\|_{L^p(M\backslash \mathcal A)}& \leq & \|F^{\ast}(-u_jd\varphi- \varphi du_j)\|_{L^p(M\backslash \mathcal A)}\nonumber \\
 &\leq &\|u_jF^{\ast}(-d\varphi)\|_{L^p(M\backslash \mathcal A)}+\||\varphi|F^{\ast}(-{\rm{sgn}}(\varphi) du_j)\|_{L^p(M\backslash \mathcal A)}\label{Lp-2} \\
 & \leq & \|u_jF^{\ast}(-d\varphi)\|_{L^p(M)}+\|\varphi\|_{L^\infty(M)}\|F^{\ast}(-{\rm{sgn}}(\varphi) du_j)\|_{L^p(M)}. \nonumber \eeqn
Observe that $\|F^{\ast}(-{\rm{sgn}}(\varphi) du_j)\|_{L^p(M)}=\|F^{\ast}(\pm du_j)\|_{L^p(M)}\rightarrow 0$ as $j\rightarrow \infty$ and
\beqn \|u_jF^{\ast}(-d\varphi)\|_{L^p(M)}=\left(\int_{{\rm{supp}}(\varphi)}u_j^pF^{\ast p}(-d\varphi)dm\right)^{1/p}\leq \|F^{\ast}(-d\varphi)\|_{L^\infty({\rm{supp}}(\varphi))}\|u_j\|_{L^p(M)}\rightarrow 0\eeqn as $j\rightarrow \infty$. Thus $\|F^{\ast}(d\varphi_j-d\varphi)\|_{L^p(M\setminus \mathcal A)}\rightarrow 0$ as $j\rightarrow \infty$. Similarly, we have $\|F^{\ast}(d\varphi -d\varphi_j)\|_{L^p(M\setminus \mathcal A)}\rightarrow 0$ as $j\rightarrow \infty$. From these and (\ref{Lp-1}), one obtains that $\|\varphi_j-\varphi\|_{1, p}\rightarrow 0$ as $j\rightarrow \infty$. The proof is completed. \end{proof}

Based on Propositions \ref{prop31}--\ref{prop33},  we prove the main result in this section, which is very important for subsequent arguments.

\begin{thm}\label{thm11}
Let $(M, F, \mathfrak{m})$ be an  $n$-dimensional forward complete Finsler measure space and $r:=r_o$ be the distance function on $(M, F)$ from some $o\in M$. Denote by $\Omega: =M\backslash\{o\}$. Let $G:(0, \sup_{\Omega} r) \rightarrow \mathbb{R}$ and $H: \mathbb{R} \rightarrow \mathbb{R}$ be  $C^1$ functions such that $G(r)\in L^{p'}_{loc}(\Omega)$,  $G'(r)\in L^{1}_{loc}(\Omega)$ and $H(0)=H'(0)=0$, where $p'=\frac{p}{p-1}$ and $1<p<n$.  Then the following inequalities hold.

 (1)  For any $u\in C^{\infty}_0(M)$, we have
\beq\label{eq0}
\int_{M}\max\{F^{\ast p}(\pm du)\}d\mathfrak{m} \geq p\int_{M}[G'(r)+G(r)\Delta r]H(u)d\mathfrak{m}-(p-1)\int_{M}|G(r)H'(u)|^{p'}d\mathfrak{m}.
\eeq

(2) For any $u\in C^{\infty}_0(M)$, we have
\beq\label{eq1}
	\int_{M}\max\{F^{\ast p}(\pm du)\}d\mathfrak{m} \geq \frac{\left|\int_{M}[G'(r)+G(r)\Delta r]H(u)d\mathfrak{m}\right|^p}{(\int_{M}|G(r)H'(u)|^{p'}d\mathfrak{m})^{p-1}}.
\eeq
provided that there exists a neighborhood $I\in \mathbb{R}$ of zero such that $H'(s)\neq 0$ for any $s\in I\backslash\{0\}$ and $G(t)\neq 0$ for $t>0$.  Further, the equality in (\ref{eq0}) or (\ref{eq1}) holds if and only if
\beq du=-{\rm{sgn}}(G(r)H'(u))|G(r)H'(u)|^{\frac 1{p-1}}dr, \ \ \  \max\left\{F^{\ast}(\pm dr)\right\}= F^{\ast}(dr) \ \ \ {\rm{on}}\ M. \label{uu-rr} \eeq
\end{thm}

\begin{proof}  Since $r$ is a Lipschitz function with  $F^{\ast}(dr)=1$ a.e. on $M$, $r$ lies in $W^{1, p}_{loc}(M)$ for $p>1$. By Proposition  \ref{prop31}, we have
\beq\int_{\Omega}\max\left\{F^{\ast p}(\pm du)\right\}d\mathfrak{m}\geq p\int_{\Omega}\left\{G'(r)+G(r)\Delta r\right\}H(u)d\mathfrak{m}-(p-1)\int_{\Omega}|G(r)H'(u)|^{p'}d\mathfrak{m},\label{3.1'}\eeq and
\beq	\int_{\Omega}\max\left\{F^{\ast p}(\pm du)\right\}d\mathfrak{m} \geq \frac{\left|\int_{\Omega}\left\{G'(r)+G(r)\Delta r\right\}H(u)d\mathfrak{m}\right|^p}{\big(\int_{\Omega}|G(r)H'(u)|^{p'}d\mathfrak{m}\big)^{p-1}}\label{3.2'} \eeq for any $u\in C_0^\infty(\Omega)$. The equality in (\ref{3.1'}) or (\ref{3.2'}) holds if and only if (\ref{u-rho}) holds.
On the other hand, since $r^{-1}(0)=\{o\}$ is compact and Cap$_p(r^{-1}(0))=0$ by Proposition \ref{prop32},  any function in $C_0^\infty(M)\subset W_0^{1, p}(M)$ can be approximated by functions in $C_0^\infty(\Omega)$ by  Proposition \ref{prop33}. Thus (\ref{eq0})--(\ref{uu-rr}) follow from the approximation arguments. \end{proof}

\section{$L^p$-uncertainty principles on Finsler measure spaces}

In this section, we shall prove the $L^p$-Heisenberg-Pauli-Weyl inequality,  $L^p$-Caffarelli-Kohn-Nirenberg interpolation inequality and $L^p$-Hardy inequality on forward complete Finsler measure spaces with radial curvatures bounded from above or below by constants by means of Theorem \ref{thm11} and Theorems \ref{thm21}-\ref{thm22}.
\subsection{$L^p$-Heisenberg-Pauli-Weyl   Principle (inequality) }

We first prove the $L^p$-Heisenberg-Pauli-Weyl inequalities, which are generalizations of the classical Heisenberg-Pauli-Weyl principle in both Euclidean and Riemannian settings (\cite{We}, \cite{Kr}).

\begin{prop}\label{prop41}
Let $(M,F,\mathfrak{m})$ be an $n(\geq 2)$-dimensional forward complete and noncompact Finsler measure space and $r= d_F(o, \cdot)$ be the distance function from some $o\in M$. Assume that $\mathbf K(\nabla r)\leq\kappa$ and $\mathbf S(\nabla r)\leq(n-1)h$ for some $\kappa, h\in\mathbb{R}$. Then, for any $-p+1<\alpha\leq 1<p<n$ and $u\in C_0^\infty(M)$, we have
\beq
	& \Big(\int_{M}\max\{F^{\ast p}(\pm du)\}d\mathfrak{m}\Big)^{\frac{1}{p}} \left(\int_{M} r^{p'\alpha} |u|^p d\mathfrak{m}\right)^{\frac{1}{p'}} \nonumber \\ & \geq \frac{n+\alpha-1}{p} \left|\int_{M}\left(1+\frac{n-1}{n+\alpha-1}\zeta_{\kappa, h}(r)\right)r^{\alpha-1}|u|^p d\mathfrak{m}\right|,\label{4.1}
\eeq
where $\zeta_{\kappa, h}(r)=r(\mathfrak{ct}_{\kappa}(r)-h)-1$. In particular, the constant $\frac{n+\alpha-1}{p}$ is sharp if $F^{\ast}(-dr)\leq F^{\ast}(dr)$. Further, the equality in (\ref{4.1}) holds if and only if
\beq  |u|=ce^{-\frac 1{\beta}r^\beta}, \ \ \ F^{\ast}(-dr)\leq F^{\ast}(dr), \ \ \ {\mathbf K}(\nabla r, \cdot)=\kappa, \ \ \ {\mathbf S}(\nabla r)=(n-1)h, \label{ur-KS}\eeq where $c$ and $\beta=\frac{p+\alpha-1}{p-1}$ are positive constants. \end{prop}

\begin{proof} Let
$$ G(s)=s^{\alpha} ,\quad H(t)=\frac 1p|t|^p$$ for $s>0$ and $t\in \mathbb R$. For any $0<\delta<\mathfrak i_o$, let $(r, y)$ be the polar coordinate system around $o$, where $r<\delta$ and $y\in S_oM$. For any compact domain $\Omega'\subset\subset \Omega=M\backslash\{o\}$, we assume that $\Omega'\subset B_\delta^+(o)$ without loss of generality (otherwise we use the union of finite geodesic balls instead of $B_\delta^+(o)$ and the arguments are similar). Since $p'\alpha+n-1>-1$ by the assumptions on $p, \alpha$ and $n$,  it follows from (\ref{hat-sigma-est}) that
$$\int_{\Omega'}G(r)^{p'}dm =\int_{\Omega'}r^{p'\alpha}\hat \sigma_o(r,y)dr d\nu_o(y)\leq (1+\varepsilon(\delta)) \mathfrak{L}_{\mathfrak{m}}(o)\int_0^\delta r^{p'\alpha+n-1}dr<+\infty, $$  which means that
$G(r)\in L^{p'}_{loc}(\Omega)$. Similarly, $G'(r)=\alpha r^{\alpha-1}\in L^1_{loc}(\Omega)$. It is obvious that $H\in C^1(\mathbb R)$ with $H(0)=H'(0)=0$. It follows from Theorem \ref{thm11}(2) and Theorem \ref{thm21}  that
\beq
 & & \Big(\int_{M}\max\{F^{{\ast}p}(\pm du)\}d\mathfrak{m}\Big)^{\frac{1}{p}} \left(\int_{M}r^{p'\alpha}|u|^p d\mathfrak{m}\right)^{\frac{1}{p'}}\nonumber \\ & &\geq\left|\int_{M}\Big(\alpha r^{\alpha-1}+r^{\alpha}(n-1)(\mathfrak{ct}_{\kappa}(r)-h)\Big)\frac{|u|^p}{p} d\mathfrak{m}\right| \nonumber \\
& & = \frac{n+\alpha-1}{p} \left|\int_{M}\left(1+\frac{n-1}{n+\alpha-1}\zeta_{\kappa, h}(r)\right)r^{\alpha-1}|u|^pd\mathfrak{m} \right| \label{uuu}
\eeq for any $u\in C_0^\infty(M)$. Thus (\ref{4.1}) follows.

In the sequel, we show that the constant $\frac{n+\alpha-1}{p}$ in (\ref{4.1}) is sharp if $F^{\ast}(-dr)\leq F^{\ast}(dr)$.
We assume by contradiction that there exists a constant $C>\frac{n+\alpha-1}{p}$ such that
\beq
\left(\int_{M}\max{F^{{\ast}p}(\pm du)}d\mathfrak{m} \right)^{\frac{1}{p}} \left(\int_{M}r^{p'\alpha}|u|^p d\mathfrak{m}\right)^{\frac{1}{p'}}\geq C \left|\int_{M}\left(1+\frac{n-1}{n+\alpha-1}\zeta_{\kappa, h}(r)\right)r^{\alpha-1}|u|^pd\mathfrak{m}\right|. \label{ineq-uuu}
\eeq
Let $(r, y)$ be the polar coordinate system around $o$, where $r(x)=r_o(x)$ for any $x\in B_\delta^+(o)$  and $y\in S_oM$.   For any $ f\in C^{\infty}_0(B^+_{\delta}(o))$, it follows from $(\ref{ineq-uuu})$ and (\ref{hat-sigma-est}) that
\beq & &\left(\int_{S_{o}M} e^{-\tau(y)}d\nu_{o}(y) \int_0^\delta \max\left\{F^{{\ast}p}(\pm df)\right\} r^{n-1} dr \right)^{\frac{1}{p}} \left( \int_{S_{o}M} e^{-\tau(y)} d\nu_{o}(y) \int_0^\delta r^{n+p'\alpha-1}|f|^p dr \right)^{\frac{1}{p'}}  \nonumber\\
& &\geq C_{\varepsilon} \left|\int_{S_{o}M} e^{-\tau(y)} d\nu_{o}(y) \int_0^\delta \left(1+\frac{n-1}{n+\alpha-1}\zeta_{\kappa, h}(r)\right)r^{n+\alpha-2}|f|^p dr\right|,\label{ineq-2}\eeq  where $$C_{\varepsilon}: =C \left( \frac{1 - \varepsilon(\delta)}{1 + \varepsilon(\delta)}\right) >\frac{n+\alpha-1}{p}$$ for sufficiently small $\delta>0$.

For any $u \in C_0^\infty(M)$, choose a sufficiently large $R > 0$ such that
$$f(r, y) := u(Rr)\in C_0^\infty(B_{\delta}^+(o))$$ is a radial function.
Let $s: =Rr$ and set $x':=\exp_o(sy)\in M$ by the forward completeness of $(M, F)$. Then $s=d_F(o, x')$ satisfies $F^{\ast}(x', ds)=1$  a.e. on $M$. Thus, $F^{\ast}(\pm df(r, y))=RF^{\ast}(\pm du(s))$. Obviously, $\int_{S_{o}M} e^{-\tau(y)}d\nu_{o}(y)>0$. From these and (\ref{ineq-2}), one obtains
\beq & &\left(\int_0^{R\delta}\max\left\{F^{{\ast}p}(\pm du)\right\} s^{n-1} ds \right)^{\frac{1}{p}} \left(\int_0^{R\delta} s^{n+p'\alpha-1}|u|^p ds \right)^{\frac{1}{p'}}  \nonumber\\
& &\geq C_{\varepsilon} \left|\int_0^{R\delta}\left(1+\frac{n-1}{n+\alpha-1}\zeta_{\kappa, h}\left(\frac sR\right)\right)s^{n+\alpha-2}|u|^p ds\right|.\label{ineq-3}\eeq
Note that $\zeta_{\kappa, h}(r)=r(\mathfrak{ct}_\kappa(r)-h)-1\rightarrow 0$ as $r\rightarrow 0$. Thus, $\zeta_{\kappa, h}\left(\frac{s}{R}\right)\rightarrow 0$ as $R\rightarrow\infty.$ Letting $R\rightarrow\infty$ in (\ref{ineq-3}) yields
\beq
\left(\int_0^{\infty} \max\left\{F^{{\ast}p}(\pm du)\right\} r^{n-1}dr  \right)^{\frac{1}{p}} \left(\int_0^{\infty} r^{n+p'\alpha-1}|u|^p dr \right)^{\frac{1}{p'}}
\geq C_{\varepsilon}\int_0^{\infty}r^{n+\alpha-2}|u|^p dr. \label{ineq-4}
\eeq

Let $\beta: =\frac{p+\alpha-1}{p-1}>0$. We now consider the test-function $u(r)=e^{-\frac 1pr^\beta}.$  Note that $C_0^\infty(M)\subset$ $ {\rm{Lip}}(M)$. Thus $r$ and hence $u$ can be approximated by functions in $C_0^\infty(M)$.  Since
$F^{\ast}(-dr)\leq F^{\ast}(dr)$, we have $\max\{F^{{\ast}p}(\pm du)\}r^{n-1}=\left({\beta}/p\right)^p r^{n+p'\alpha-1}u^p.$ From this and (\ref{ineq-4}), one obtains
\beq C_\varepsilon \leq \frac{{\beta}\int_0^\infty  r^{n+p'\alpha-1} e^{-r^{\beta}} dr}{p\int_0^\infty  r^{n+\alpha-2} e^{-r^{\beta}} dr}.\label{C-epsilon}\eeq
Recall that the Gamma function $\Gamma(s)=\int_0^\infty t^{s-1}e^{-t}dt<\infty$ for $s>0$ and $\Gamma(s+1)=s\Gamma(s)$. Then
\beq \int_0^\infty e^{-s^\beta} s^\theta ds= \frac{1}{\beta} \Gamma \left(\frac{\theta+1}{\beta}\right),  \ \ \ \   \forall\beta>0, \ \ \theta>-1.\label{Gamma}\eeq
By the assumptions on $p, \alpha$ and $n$, both $n+p'\alpha-1$ and $n+\alpha-2$ are greater than $-1$, and $\beta >0$.  A direct calculation with (\ref{Gamma}) shows that the RHS of (\ref{C-epsilon}) is equal to $\frac{n+\alpha-1}p$. Thus $C_\varepsilon\leq\frac{n+\alpha-1}{p}$. This contradicts the assumption. Hence the constant $\frac{n+\alpha-1}{p}$ is sharp.

Further, by Theorem \ref{thm11}(2) and the choices of $G(r), H(u)$, the equality in (\ref{4.1}) holds if and only if the equality in (\ref{uuu}) holds if and only if $u$ and $r$ satisfy
\beq du=- r^{\frac{\alpha}{p-1}}udr, \ \ \max\left\{F^{\ast}(\pm dr)\right\}= F^{\ast}(dr), \ \  \Delta r=(n-1)\big(\mathfrak{ct}_\kappa(r)-h\big). \label{u-rho*} \eeq
In (\ref{u-rho*}), solving the first equation yields $|u|=ce^{-\frac 1 \beta r^\beta}$ for some constant $c>0$. The second one implies that $F^{\ast}(-dr)\leq F^{\ast}(dr)$ and the third one implies that $\mathbf K(\nabla r)=\kappa$ and $\mathbf S(\nabla r)=(n-1)h$ by Theorem \ref{thm21}. This finishes the proof.
\end{proof}

Based on Proposition \ref{prop41}, we prove the following result.

\begin{prop}\label{prop42}
Let $(M, F, \mathfrak{m})$ be an $n(\geq 2)$-dimensional forward complete Finsler measure space and there exist some $o\in M$ such that $\Lambda=\Lambda_o$. Assume that ${\mathbf{Ric}}(\nabla r)\geq(n-1)\kappa$ and ${\mathbf S}(\nabla r)\geq(n-1)h$ for some $\kappa, h\in\mathbb{R}$, where   $r:= d_F(o, \cdot)$.  Then the following statements are equivalent.

{\rm (1)}  For $-p+1<\alpha\leq 1<p<n$ and $\kappa, h\in\mathbb{R}$ satisfying $h\geq 0$ if $\kappa=0$ and $h\geq\sqrt{-\kappa}>0$ if $\kappa<0$, it holds
\beq
	 & &\left(\int_{M}\min\left\{F^{\ast p}(\pm du)\right\}d\mathfrak{m}\right)^{\frac{1}{p}} \left(\int_{M} r^{p'\alpha} |u|^p d\mathfrak{m}\right)^{\frac{1}{p'}}  \nonumber \\
& &\geq \frac{n+\alpha-1}{p} \left|\int_{M}\left(1+\frac{n-1}{n+\alpha-1}\zeta_{\kappa, h}(r)\right)r^{\alpha-1}|u|^p d\mathfrak{m}\right|, \ \ \ \ \ \forall u\in C_0^\infty(M),\label{Ric-S}
\eeq
where $\zeta_{\kappa, h}(r)$ is defined as in Proposition \ref{prop41}.

{\rm (2)}  $\Lambda=1$,  ${\mathbf K}(\nabla r, \cdot)=\kappa$ and ${\mathbf S}(\nabla r)=(n-1)h$. \end{prop}

\begin{proof}  (2) $\Rightarrow$ (1) follows directly from $\Lambda=1$ and Proposition \ref{prop41}. Next we prove (1) $\Rightarrow (2)$.
 Since (1) holds for any $u\in C_0^\infty(M)$, we may choose $u=e^{-\frac 1p r^{\beta}}$ in (\ref{Ric-S}) as in Proposition \ref{prop41}, where $\beta=\frac{p+\alpha-1}{p-1} > 0$. Obviously, $u$ can be approximated by functions in $C_0^\infty(M)$. A direct calculation shows
\beq \int_M \min\{F^{\ast p}(\pm du)\}d\mathfrak{m}\leq \left(\frac{\beta}{p}\right)^p\int_M r^{p'\alpha}e^{-r^\beta}d\mathfrak{m}\label{u-r}\eeq if the integral of RHS is well defined. Now we check that the  RHS of (\ref{u-r}) is indeed well defined. For this, set
$$ \mathfrak{I}:=\int_M r^{p'\alpha}e^{-r^\beta}d\mathfrak{m}. $$

{\it Case 1.} $\kappa>0$. In this case, $(M,F)$ is compact by Myers' Theorem and hence  $\mathfrak{I}$ is well defined.

{\it Case 2.} $\kappa\leq0$. It is obvious that $\mathfrak I<\infty$ in the compact case. Now we consider the noncompact case.  By $(\ref{co-area})$, one obtains
\begin{eqnarray}\label{4.2.3}
\mathfrak{I}=\int_0^\infty dt\int_{S_{t}(o)}t^{p'\alpha}e^{-t^\beta}dA = \int_0^\infty t^{p'\alpha}e^{-t^\beta}A(S_t(o))dt,
\end{eqnarray}
where $S_{t}(o):=\{x\in M|r(x)=t\}$ and $A(S_o(t))$ means the area of $S_t(o)$ induced by $\mathfrak m$.
Similar to the proof of (\ref{hat-sigma*}), we have
$$\hat{\sigma}_o(t, y)\leq e^{-\tau(y)}(e^{-ht}\mathfrak{s}_{\kappa}(t))^{n-1},\quad \forall\ 0 < t \leq \min\{\mathfrak{i}_{o}, \infty\},$$ which implies that
\beqn
A(S_t(o))=\int_{S_{o} M}\hat{\sigma}_o(t, y)d\nu_{o}(y)\leq \int_{S_{o} M}e^{-\tau(y)}(e^{-ht}\mathfrak{s}_{\kappa}(t))^{n-1}d\nu_{o}(y)=\mathfrak{L}_{\mathfrak{m}}(o)(e^{-ht}\mathfrak{s}_{\kappa}(t))^{n-1}.
\eeqn  Plugging this in $(\ref{4.2.3})$ yields
 \beqn
\mathfrak{I}&\leq &\mathfrak{L}_{\mathfrak{m}}(o)\int_0^{\infty}t^{p'\alpha}e^{-t^\beta}(e^{-ht}\mathfrak{s}_{\kappa}(t))^{n-1}dt\nonumber \\
&=&\mathfrak{L}_{\mathfrak{m}}(o)\left\{\int_0^{\delta}t^{p'\alpha}e^{-t^\beta}(e^{-ht}\mathfrak{s}_{\kappa}(t))^{n-1}dt
+\int_\delta^{\infty}t^{p'\alpha}e^{-t^\beta}(e^{-ht}\mathfrak{s}_{\kappa}(t))^{n-1}dt\right\}\nonumber \\
&:=&\mathfrak{L}_{\mathfrak{m}}(o)\left\{I+II \right\}\label{I-II}
\eeqn for any $0<\delta<\infty$. To prove that $\mathfrak I$ is well defined, it suffices to prove that $I<\infty$ and $II<\infty$ respectively  since $\mathfrak{L}_{\mathfrak{m}}(o)$ is finite.

Note that $\mathfrak s_\kappa(t)\sim t$ as $t\rightarrow 0$. There exists a sufficiently small $\delta>0$ and a constant $c_1$ such that  $(e^{-ht}\mathfrak{s}_{\kappa}(t))^{n-1}\leq c_1 t^{n-1}$ for any $t\in (0, \delta]$. Since $\alpha>-p+1$ and $n>p$, we have  $p'\alpha+n>-p+n>0$. Thus
 \beqn
I\leq c_1\int_0^{\delta}t^{p'\alpha+n-1}e^{-t^\beta}dt < c_1\int_0^{\delta}t^{p'\alpha+n-1}dt=\frac{c_1}{p'\alpha+n}\delta^{p'\alpha+n}<\infty. \eeqn

For $t\geq\delta$, we may assume that $\delta$ is sufficiently large without loss of generality. Otherwise we argue by dividing $[\delta, +\infty)$ into several subintervals.

 {\it Subcase 1.} $\kappa=0$. In this case, $\mathfrak{s}_\kappa(t)=t$ and $(e^{-ht}\mathfrak{s}_{\kappa}(t))^{n-1}=t^{n-1}e^{-h(n-1)t}$. Then
 \beqn
II=\int_\delta^{\infty}t^{p'\alpha+n-1}e^{-t^\beta-h(n-1)t}dt.\eeqn
 Recall that $\alpha$ satisfies $-p+1<\alpha\leq 1$ and the Gamma function $\Gamma(s)=\int_0^\infty t^{s-1}e^{-t}dt$ converges for any $s>0$.  When $0<\alpha<1$, we have $\beta>1$.  Since $h\geq 0$,  we have
 %\beq  \int_{0}^\infty t^{\gamma}e^{-at^\beta}dt=\frac 1\beta a^{-\frac{\gamma+1}{\beta}}\int_0^\infty s^{\frac{\gamma+1}{\beta}-1}e^{-s}ds=\frac 1\beta a^{-\frac{\gamma+1}{\beta}}\Gamma\left(\frac {\gamma+1}{\beta}\right)<\infty,\label{Gamma}\eeq for any constants $a>0$ and $\gamma>-1$.  Let $\gamma:=p'\alpha+n-1>-1$. From (\ref{Gamma}), one obtains
 \beqn II \leq \int_{\delta}^\infty t^{p'\alpha+n-1}e^{-t^\beta}dt \leq\int_{0}^\infty t^{p'\alpha+n-1}e^{-t^\beta}dt=\frac 1{\beta} \Gamma\left(\frac {p'\alpha+n}{\beta}\right) < \infty.  \eeqn
  When $\alpha=0$, we have $\beta=1$ and $1+(n-1)h\geq 1$ by the assumption. Hence,
  \beqn II\leq\int_0^{\infty}t^{n-1}e^{-(1+h(n-1))t}dt=\frac{1}{(1+(n-1)h)^n}\Gamma(n) < \infty.\eeqn
  When $-p+1<\alpha<0$, we have $0<\beta<1$. Note that $e^{-h(n-1)t}\leq 1$ since $h\geq0$. Thus $II\leq \int_\delta^{\infty}t^{p'\alpha+n-1}e^{-t^\beta}dt<\infty$ as in the case when $\alpha>0$.

  {\it Subcase 2.}  $\kappa<0$. In this case, we have $\mathfrak{s}_\kappa(t)=\frac{\sinh(\sqrt{-\kappa}t)}{\sqrt{-\kappa}}\leq \frac{e^{\sqrt{-\kappa}t}}{\sqrt{-\kappa}}$ and $\mu: =(n-1)(\sqrt{-\kappa}-h)\leq 0$ by the assumption.  This means that  $(e^{-ht}\mathfrak{s}_{\kappa}(t))^{n-1}\leq c_2 e^{\mu t}\leq c_2$, where $c_2: =(-\kappa)^{-(n-1)/2}>0$. By a variable substitution, we have
 \beqn
II\leq c_2\int_\delta^{\infty}t^{p'\alpha}e^{-t^\beta}dt=\frac{c_2}{\beta}\int_{\delta^\beta}^\infty s^{\frac{p'\alpha+1-\beta}{\beta}}e^{-s}ds=\frac{c_2}{\beta}\int_{\delta^\beta}^\infty s^{\frac{\alpha}{\beta}}e^{-s}ds.
\eeqn
 Since $\beta>0$ and $s\geq \delta^{\beta} >0$, it is obvious that $s^{\frac \alpha\beta}<e^{s/2}$ when $\alpha\geq 0$. Thus $II< \frac{c_2}{\beta}\int_{\delta^\beta}^\infty e^{-s/2}ds=\frac{2}{\beta}c_2 e^{-\frac 12 \delta^\beta} <\infty$. For the case when $\alpha<0$, we have $s^{\frac \alpha\beta}<\delta^{\alpha}$. Thus  $II\leq \frac{c_2}{\beta}\delta^\alpha\int_{\delta^\beta}^\infty e^{-s}ds <\infty$.  Combining Subcase 1 with Subcase 2 yields $II<\infty$ in case 2. Thus $\mathfrak I<\infty$.

Note that $n+\alpha-1>p+\alpha-1>0$ and $\alpha+\beta-1=p'\alpha$. By (\ref{Ric-S}) with $u=e^{-\frac 1p r^{\beta}}$ and (\ref{u-r}), one obtains
\beq
\beta\int_{M}r^{p'\alpha}e^{-r^\beta}d\mathfrak{m}
&\geq& p\left(\int_{M}\min\{F^{\ast p}(\pm du)\}d\mathfrak{m}\right)^{\frac{1}{p}} \left(\int_{M} \rho^{p'\alpha} e^{-r^\beta} d\mathfrak{m}\right)^{\frac{1}{p'}}\nonumber\\
% &\geq &(n+\alpha-1)\left|\int_{M}\left(1+\frac{n-1}{n+\alpha-1}\zeta_{\kappa, h}(r)\right)r^{\alpha-1}e^{-r^\beta} d\mathfrak{m}\right| \nonumber\\
&\geq &\left|\int_{M}\Big(n+\alpha-1+(n-1)\zeta_{\kappa, h}(r)\Big)r^{\alpha-1}e^{-r^\beta} d\mathfrak{m}\right| \nonumber\\
&=& \left|\int_{M}\Big(\alpha+(n-1)r(\mathbf{ct}_\kappa(r)-h)\Big)r^{\alpha-1}e^{-r^\beta} d\mathfrak{m}\right| \nonumber\\
&\geq & \left|\alpha\int_{M}r^{\alpha-1}e^{-r^\beta} d\mathfrak{m}+\int_{M}r^{\alpha}e^{-r^\beta}\Delta r d\mathfrak{m}\right| \nonumber\\
&=& \beta \int_{M}r^{\alpha+\beta-1}e^{-r^\beta} d\mathfrak{m}= \beta \int_{M}r^{p'\alpha}e^{-r^\beta} d\mathfrak{m}.\label{rr}
\eeq
where we used Theorem \ref{thm22} in the third inequality and (\ref{Laplacian}) in the second equality. This inequality implies that all inequalities become equalities. In particular, we have that
\beqn \min\{F^{\ast p}(\pm du)\}=F^{{\ast}p}(-du) =\left({\beta}/{p}\right)^p r^{p'\alpha}e^{-r^\beta}F^{{\ast}p}(dr), \ \ \ \ \Delta r=(n-1)(\mathbf{ct}_\kappa(r)-h),\eeqn  which are equivalent to  $F^{\ast}(dr)\leq F^{\ast}(-dr)$, $\mathbf K(\nabla r, \cdot)=\kappa$ and $S(\nabla r)=(n-1)h$.
For any $y\in S_oM$, let $\gamma :[0, r]\rightarrow M$ be a normal minimal geodesic with $\gamma(0)=o$ and $\dot\gamma(0)=y$. Then $\dot\gamma(t)=\nabla r$. Note that $dr=\mathcal L(\nabla r)$ (see \S 2.3). Then $F^{\ast}(\mathcal L(\nabla r))\leq  F^{\ast}(-\mathcal L(\nabla r))$, i.e., $F^{\ast}(\mathcal L(\dot\gamma(t)))\leq  F^{\ast}(-\mathcal L(\dot\gamma(t)))$. Letting $t\rightarrow 0^+$ yields $F^{\ast}(o, \mathcal L(y))\leq F^{\ast}(o, -\mathcal L(y))$. Since $y$ is arbitrary, substituting $y$ with $-y$ yields that $F^{\ast}(o, -\mathcal L(y))\leq F^{\ast}(o, \mathcal L(y))$. Thus $\Lambda^{\ast}_o=1$, equivalently, $\Lambda_o=1$, and hence $\Lambda=1$ by the assumption.  (2) follows.
\end{proof}

\begin{rem} {\rm In (1) of Proposition \ref{prop42}, we assume that $h\geq 0$ when $\kappa=0$ for simplicity. In fact, the condition ``$h\geq 0$" is not necessary when $\kappa=0$ and $0<\alpha<1$. This is because $\beta>1$ and $-t^\beta-h(n-1)t\leq -t^\beta+|h|(n-1)t < -\frac{1}{2}t^\beta$ as $t$ is sufficiently large. By a variable substitution $\bar t:=\frac 12 t^\beta$, we have $II< \int_{0}^\infty t^{p'\alpha+n-1}e^{-\frac{1}{2}t^\beta}dt \leq\beta^{-1} 2^{\frac{p'\alpha+n}{\beta}}\Gamma\left(\frac {p'\alpha+n}{\beta}\right) < \infty.$} \end{rem}

\subsection{$L^p$-Caffarelli-Kohn-Nirenberg interpolation inequality}

Next we prove the $L^p$-Caffarelli-Kohn-Nirenberg interpolation inequalities, which extends the corresponding ones in both Euclidean spaces (\cite{CKN}) and Finsler spaces (\cite{HKZ}) when $p=2$.
\begin{prop}\label{prop43} Under the same assumptions as in Proposition \ref{prop41},  we have
\beq
& &	\left(\int_{M}\max\{F^{\ast p}(\pm du)\}d\mathfrak{m}\right)^{\frac{1}{p}} \left(\int_{M} r^{p'\alpha} |u|^{p'(q-1)} d\mathfrak{m}\right)^{\frac{1}{p'}}\nonumber \\
& & \geq \frac{n+\alpha-1}{q} \left|\int_{M}\left(1+\frac{n-1}{n+\alpha-1}\zeta_{\kappa, h}(r)\right)r^{\alpha-1}|u|^q d\mathfrak{m}\right|, \ \ \ \forall u\in C^{\infty}_0(M) \label{uu-r}\eeq for $1< p< q$, $-p+1< \alpha$ and $0< q(n-p)< p(n+\alpha-1)$, where $\zeta_{\kappa, h}(r)$ is defined as in Proposition \ref{prop41}. In particular, the constant $\frac{n+\alpha-1}{q}$ is sharp if $F^{\ast}(-dr)\leq F^{\ast}(dr)$.
 Further, the equality in (\ref{uu-r}) holds if and only if
 \beq  |u|=\left(\beta\vartheta\right)^\vartheta \left(r^\beta+ c\right)^{-\vartheta}, \ \  F^{\ast}(-dr)\leq F^{\ast}(dr), \ \  {\mathbf K}(\nabla r, \cdot)=\kappa, \ \  {\mathbf S}(\nabla r)=(n-1)h, \label{ur-KS*}\eeq where $\beta=\frac{p+\alpha-1}{p-1}$, $\vartheta= \frac{p-1}{q-p}$ and $c$ is a constant with $r^\beta+c>0$. \end{prop}
\begin{proof} The proof is similar to that of Proposition \ref{prop41}. The difference lies in the choices of the functions $G$ and $H$ in Theorem \ref{thm11}. Let
$$G(s)=s^{\alpha}, \quad H(t)=\frac{|t|^q}{q}$$ for $s>0$ and $t\in \mathbb R$. Then $G$ and $H$ satisfy the assumptions in Theorem \ref{thm11} as in the proof of Proposition \ref{prop41}. Then (\ref{uu-r}) follows from Theorem \ref{thm11}(2) and Theorem \ref{thm21} for any $u\in C_0^\infty(M)$.

Next we show that the constant $\frac{n+\alpha-1}{p}$ in (\ref{4.1}) is sharp if $F^{\ast}(-dr)\leq F^{\ast}(dr)$.
We assume by contradiction that there exists a constant $C>\frac{n+\alpha-1}{q}$ such that
\beqn
\left(\int_{M}\max{F^{{\ast}p}(\pm du)}d\mathfrak{m} \right)^{\frac{1}{p}} \left(\int_{M}r^{p'\alpha}|u|^{p'(q-1)} d\mathfrak{m}\right)^{\frac{1}{p'}}\geq C \left|\int_{M}\left(1+\frac{n-1}{n+\alpha-1}\zeta_{\kappa, h}(r)\right)r^{\alpha-1}|u|^qd\mathfrak{m}\right|  \label{ineq-1} \eeqn for any $u\in C_0^\infty(M)$. By the same arguments as in the proof of Proposition \ref{prop41}, there is a positive constant $C_\varepsilon>\frac{n+\alpha-1}{q}$ such that
\beq
\left(\int_0^{\infty} \max\left\{F^{{\ast}p}(\pm du)\right\} r^{n-1}dr  \right)^{\frac{1}{p}} \left(\int_0^{\infty} r^{n+p'\alpha-1}|u|^{p'(q-1)} dr \right)^{\frac{1}{p'}}
\geq C_{\varepsilon}\int_0^{\infty}r^{n+\alpha-2}|u|^q dr. \label{ineq-4*}
\eeq
 For the case when $u\equiv 0$,  both sides of (\ref{ineq-4*}) are zero regardless of the choice of $C_{\varepsilon}$. We may choose a positive constant $C_\varepsilon$ such that $C_\varepsilon\leq \frac{n+\alpha-1}{q}$, which is a contradiction. Now we consider $u\in C_0^\infty(M)\backslash \{0\}$. In this case, if $F^{\ast}(-dr)\leq F^{\ast}(dr)$ and $u=u(r)$ is a radial function, then (\ref{ineq-4*}) implies that $I_1^{1/p}I_2^{1/p'}\geq C_{\varepsilon}I_3,$
 where $$I_1:= \int_0^\infty|u'(r)|^pr^{n-1}dr, \ \ \ \ I_2:= \int_0^\infty r^{n+p'\alpha-1}|u|^{p'(q-1)}dr, \ \ \ \ I_3:= \int_0^\infty r^{n+\alpha-2}|u|^{q}dr.$$
Let $\beta=\frac{p+\alpha-1}{p-1}>0$ as before and $\vartheta= \frac{p-1}{q-p}>0$. Choose $u=(1+r^\beta)^{-\vartheta}$ as a test function, which can be approximated by functions in $C_0^\infty(M)$. Recall that the Beta function $\mathcal B$ is given by
$$\mathcal B(\mathfrak p, \mathfrak q)=\int_0^\infty t^{\mathfrak p-1}(1+t)^{-(\mathfrak p+\mathfrak q)}dt=\int_0^1t^{\mathfrak p-1}(1-t)^{\mathfrak q-1}dt,  $$ which converges for any $\mathfrak p>0$ and $\mathfrak q>0$. Thus, for any $\sigma>-1$ and $\tau>\frac{\sigma+1}{\beta}$, we have
\beqn \int_0^\infty r^\sigma (1+r^\beta)^{-\tau}dr=\frac 1\beta\int_0^\infty t^{\frac{\sigma+1}{\beta}-1}(1+t)^{-\tau}dt=\frac 1\beta \mathcal B\left(\frac{\sigma+1}{\beta}, \tau-\frac{\sigma+1}{\beta}\right).\eeqn
From this, it is easy to check that
\beq I_1=\frac{(\beta\vartheta)^p}{\beta}\mathcal B(a_1, b_1), \ \ \ \ I_2=\frac 1\beta \mathcal B(a_2, b_2), \ \ \ \ I_3=\frac 1{\beta}\mathcal B( a_3, b_3), \label{III*} \eeq where $a_1=a_2=\frac{n+p'\alpha}{\beta}$, $a_3=\frac{n+\alpha-1}{\beta}$ and $b_1=b_2=\frac{p(q-1)}{q-p}-a_1$ and $b_3=q\vartheta-a_3$. In fact, by the assumptions on $p, q, \alpha$, we have $a_1>(n-p)/\beta >0$, $a_3>(n-p)/\beta >0$ and $a_1=a_3+1$. Consequently,
 \beqn b_1=\frac{p(q-1)}{q-p}-a_1=\frac{q(p-1)}{q-p}-a_3=q\vartheta-a_3= b_3 >0,\eeqn where the positivity follows from the assumption that  $p(n+\alpha-1)>q(n+p)$. From these and (\ref{III*}), one obtains
 \beqn C_\varepsilon &\leq & \frac{I_1^{1/p}I_2^{1/p'}}{I_3}=\beta\gamma\cdot \frac{\mathcal B(a_3+1, b_1)}{\mathcal B(a_3, b_1)}=\beta\gamma \cdot \frac{\Gamma(a_3+1)\Gamma(b_1)}{\Gamma(a_3+ b_1+1)}\cdot \frac{\Gamma(a_3+b_1)}{\Gamma(a_3)\Gamma(b_1)}\nonumber \\
 & = & \beta\vartheta \cdot \frac{a_3}{a_3+b_1}=\frac{n+\alpha-1}q,\eeqn which is a contradiction. This finishes the sharpness of $\frac{n+\alpha-1}q$.

Further, by Theorem \ref{thm11}, the equality in (\ref{uu-r}) holds if and only if $u$ and $r$ satisfy
\beq d|u|=- r^{\frac{\alpha}{p-1}}|u|^{\frac{q-1}{p-1}}dr, \ \ \max\left\{F^{\ast}(\pm dr)\right\}= F^{\ast}(dr), \ \  \Delta r=(n-1)\big(\mathfrak{ct}_\kappa(r)-h\big). \label{u-rho**} \eeq
Solving (\ref{u-rho**})$_1$  yields $|u|=\left(\beta\vartheta\right)^\vartheta \left(r^\beta+c\right)^{-\vartheta}$ for some $c\in \mathbb R$ with $r^\beta+c>0$. The rest follows from the proof of Proposition \ref{prop41}. The proof is finished.
\end{proof}

Based on Proposition \ref{prop42}, we obtain the following result.

\begin{prop}\label{prop44}Under the same assumptions as in Proposition \ref{prop42}, the following statements are equivalent.

{\rm (1)} For $1< p <q$, $-p+1< \alpha$,  $0< q(n-p)< p(n+\alpha-1)$, and $\kappa, h\in\mathbb{R}$ satisfying $h\geq 0$ if $\kappa=0$ and  $h\geq\sqrt{-\kappa}>0$ if $\kappa<0$, it holds
\beq
	&  &\left(\int_{M}\min\{F^{\ast p}(\pm du)\}d\mathfrak{m}\right)^{\frac{1}{p}} \left(\int_{M} r^{p'\alpha} |u|^{p'(q-1)} d\mathfrak{m}\right)^{\frac{1}{p'}}\nonumber \\
  &  &\geq \frac{n+\alpha-1}{q} \left|\int_{M}\left(1+\frac{n-1}{n+\alpha-1}\zeta_{\kappa, h}(r)\right)r^{\alpha-1}|u|^q d\mathfrak{m}\right|, \ \ \ \ \ \forall u\in C_0^\infty(M),  \label{uu-r*}
\eeq  where $\zeta_{\kappa, h}(r)$ is defined as in Proposition \ref{prop41}.

{\rm (2)}  $\Lambda=1$,  ${\mathbf K}(\nabla r, \cdot)=\kappa$ and ${\mathbf S}(\nabla r)=(n-1)h$.
\end{prop}
\begin{proof} (2) $\Rightarrow $ (1) follows directly from $\Lambda=1$ and Proposition \ref{prop43}. The proof of (1) $\Rightarrow (2)$  is similar to that of Proposition \ref{prop42}. The difference lies in the choice of test function  $u=(1+r^{\beta})^{-\vartheta}$ instead of  $u= e^{-\frac 1p r^{\beta}}$ in Proposition \ref{prop42}, where $\beta=\frac{p+\alpha-1}{p-1}>0$ and $\vartheta =\frac{p-1}{q-p}>0$ as before. In this case, we have
\beqn \int_M\min\{F^{\ast p}(\pm du)\}d\mathfrak m\leq \int_M|u'(r)|^pd\mathfrak m=(\beta\vartheta)^p\int_Mr^{p'\alpha} u^{p'(q-1)}d\mathfrak m.\eeqn
Let $I:=\int_M r^{p'\alpha} u^{p'(q-1)}d\mathfrak m$. By the assumptions on $p, q, \alpha$ and $h, k$ in (1), one obtains that $I<+\infty$, whose proof is similar to that of Proposition (\ref{prop42}). Thus the conclusion follows from the same arguments as in (\ref{rr}).\end{proof}

 With Propositions \ref{prop41}--\ref{prop44}, we begin to prove Theorems \ref{thm12}--\ref{thm13}.

\begin{proof} [Proof of Theorem \ref{thm12}] For any $x\in M$, let $r_x=d_F(x, \cdot)$ and $\gamma: [0, \ell]\rightarrow M$
 be a normal minimal geodesic from $x=\gamma(0)$, where $\ell\leq \min\{\mathfrak i_x, \pi/\sqrt{\kappa}\}$. Thus $\gamma(t)$ is not the cut point of $x$ for $t\in [0, \ell]$ and hence $\dot\gamma(t)=\nabla r_x(\gamma(t))$ with $F(\dot\gamma)=1$. Consequently, $\mathbf K(\nabla r_x)\leq \kappa$ and $\mathbf S(\nabla r_x)\leq (n-1)h$ from the assumption. By Propositions \ref{prop41} and \ref{prop43},  (\ref{U-max}) follows. Next we prove that the equivalence of (1)--(3).

 (1) $\Rightarrow$ (2). Obvious.

 (2) $\Rightarrow$ (3).  Assume that (2) holds. By Propositions \ref{prop41} and \ref{prop43}, we have (\ref{4.1}) in Case (I) and (\ref{uu-r}) in Case (II), in which the equalities hold if and only if $u$ is given by (\ref{u-func}),  $F^{\ast}(-dr_o)\leq F^{\ast}(dr_o)$, $\mathbf K(\nabla r_o)=\kappa$ and $\mathbf S(\nabla r_o)=(n-1)h$. For any $y\in S_oM$, let $\gamma :[0, r_o]\rightarrow M$ be a normal minimal geodesic with $\gamma(0)=o$ and $\dot\gamma(0)=y$. Then $\dot\gamma(t)=\nabla r_o$ and $F^\ast(-\mathcal L(\dot\gamma(t)))\leq F^\ast(\mathcal L(\dot\gamma(t)))$ as in the proof of Proposition \ref{prop42}. Letting $t\rightarrow 0^+$ yields $F^\ast(o, -\mathcal L(y))\leq F^\ast(o, \mathcal L(y))$. Since $y$ is arbitrary, substituting $y$ with $-y$ yields $F^\ast(o, \mathcal L(y))\leq F^\ast(o, -\mathcal L(y))$. Consequently,  $\Lambda^\ast_o=1$, equivalently, $\Lambda_o=1$, which means that $\Lambda=1$ by (\ref{LL})$_1$, i.e., $F$ is reversible. Moreover, $m(B_{r}^+(o))=\mathcal{V}_{o,\kappa, h, n}(r)$ by Corollary \ref{cor33}. From this, (\ref{LL})$_2$ and Corollary \ref{cor31}, we get (3).

 (3) $\Rightarrow$ (1). Since $F$ is reversible, by Propositions \ref{prop41} and \ref{prop43}, the constant $\frac{n+\alpha-1}{q}$ is sharp in (\ref{U-max}) for every $x\in M$. Moreover, the function $u$ given by (\ref{u-func}) satisfies the equality in (\ref{U-max}), i.e., (1) holds. This finishes the proof. \end{proof}

\begin{proof} [Proof of Theorem \ref{thm13}] The proof is similar to that of Theorem \ref{thm12} based on Propositions \ref{prop42} and \ref{prop44}. We leave it to readers. \end{proof}
\subsection{$L^p$-Hardy inequality}

Finally, we prove the $L^p (p>1)$-Hardy type inequality, generalizing the corresponding ones in both Riemannian case (\cite{DD}, \cite{KKPZ}) and Finslerian case (\cite{HKZ}) when $p=2$.

\begin{prop}\label{prop45} Let $(M,F,\mathfrak{m})$ be an $n(\geq 2)$-dimensional forward complete and noncompact Finsler measure space and $r= d_F(o, \cdot)$ be the distance function from a fixed $o\in M$ but arbitrarily.
Assume that ${\mathbf K}(\nabla r)\leq\kappa$ and ${\mathbf S}(\nabla r)\leq(n-1)h$ for some $\kappa, h\in\mathbb{R}$. Then, for any $1<p<n$ and $u\in C^{\infty}_0(M)$, we have
\beq
	\left(\int_{M}\max\{F^{\ast p}(\pm du)\}d\mathfrak{m}\right)^{\frac{1}{p}} \left(\int_{M} r^{-p} |u|^{p} d\mathfrak{m}\right)^{\frac{1}{p'}}  \geq \frac{n-p}{p} \left|\int_{M}\left(1+\frac{n-1}{n-p}\zeta_{\kappa, h}(r)\right)r^{-p}|u|^p d\mathfrak{m}\right|.\label{Hardy}
\eeq
Further, if $F^{\ast}(-dr)\leq F^{\ast}(dr)$, then the constant $\frac{n-p}{p}$ is sharp but never achieved.
\end{prop}
\begin{proof} Let
 $$G(s)=s^{1-p}, \ \ \ \ H(t)=\frac 1p {|t|^p} $$ for any $s>0$ and $t\in \mathbb R$. Since $1<p<n$,  $G(r)$ and $H(t)$ satisfy the assumptions of Theorem \ref{thm11} similar to the proof in Proposition \ref{prop41}. Thus (\ref{Hardy}) follows from Theorems \ref{thm11}(2) and \ref{thm21}.

In the sequel, we show that the constant $\bar \beta:= \frac{n-p}{p}$ is sharp if $F^{\ast}(-dr)\leq F^{\ast}(dr)$.
We assume by contradiction that there exists a constant $C > \bar\beta$ such that
\beq
\left(\int_{M}\max\{F^{\ast p}(\pm du)\}d\mathfrak{m}\right)^{\frac 1p} \left(\int_{M} r^{-p} |u|^{p} d\mathfrak{m}\right)^{\frac{1}{p'}}  \geq C \left|\int_{M}\left(1+\frac{n-1}{n-p}\zeta_{\kappa, h}(r)\right)r^{-p}|u|^p d\mathfrak{m}\right|.\label{Lp-Hardy}
\eeq
For any $0<\varepsilon < R_1< R_2<\mathfrak{i}_{o}$, choose a cut-off function $\psi\in C_0^\infty(M)$ with $0<\psi<1$  such that $\psi$ is $1$ on $B_{R_1}$ and zero on $M\backslash B_{R_2}$
%\[\psi = \begin{cases} 1, & x \in B_{R_1}^+(o), \\ 0, & x \in M\backslash B_{R_2}^+(o), \end{cases}\]
with $F^{\ast}(-d\psi)<\frac 1{R_2-R_1}$ a.e. on $M$, where $B_{\bullet}: =B_{\bullet}^+(o)$.
Define $u_\varepsilon(x):=[\max\{\varepsilon, r(x)\}]^{-\bar\beta}$ and $u:=\psi u_\varepsilon\geq0$. Then $u$ can be approximated by functions in $C_0^\infty(M)$. Thus the first integral on LHS of (\ref{Lp-Hardy}) is estimated by
\beqn
\mathcal I_1 &:=&\int_{M} \max\left\{F^{{\ast}p}(\pm du)\right\}d\mathfrak{m}\\ \nonumber
&\leq &\int_{B_{R_1}\setminus B_{\varepsilon}}F^{{\ast}p}(\bar\beta r^{-\bar\beta-1}dr)d\mathfrak{m}+\int_{B_{R_2}\setminus B_{R_1}}F^{{\ast}p}(d(\psi r^{-\bar\beta}))d\mathfrak{m}\\ \nonumber
&: =&{\bar\beta}^p \mathfrak{J}_1+\mathfrak{J}_2,
\eeqn
where $\mathfrak{J}_1=\int_{B_{R_1}\setminus B_{\varepsilon}}r^{-n}d\mathfrak{m}$ and $\mathfrak{J}_2=\int_{B_{R_2}\setminus B_{R_1}}F^{{\ast}p}(d(\psi r^{-\bar\beta}))d\mathfrak{m}\geq 0$. Clearly, $\mathfrak{J}_2$ is finite and independent of $\varepsilon$. Moreover, the second integral on LHS of (\ref{Lp-Hardy}) is given by
\beqn
\mathcal I_2&:=&\int_{M} r^{-p}|u|^p d\mathfrak m \\
&=& \int_{B_{R_1}\backslash B_{\varepsilon}}r^{-n} d\mathfrak{m}+\int_{B_{R_2}\setminus B_{R_1}}|\psi|^pr^{-n}d\mathfrak{m}+\varepsilon^{-p\bar\beta}\int_{ B_{\varepsilon}}r^{-p} d\mathfrak{m}\\
&: =&\mathfrak{J}_1+\mathfrak{J}_3+\mathfrak{J}_4.
\eeqn
 Clearly, $\mathfrak{J}_3$ is finite and independent of $\varepsilon$. We now estimate $ \mathfrak{J}_1$ and $ \mathfrak{J}_4$. The co-area formula $(\ref{co-area})$ then yields
\beq
\mathfrak{J}_1=\int_\varepsilon^{R_1} dt\int_{S_{t}(o)}t^{-n}dA=\int_\varepsilon^{R_1}t^{-n}A(S_t(o))dt.\label{J1}
\eeq
Let $(r, y)$ be the polar coordinate system around $o$ such that $r=r(x)$ and $y\in S_oM$. From (\ref{sigma-limit}), there exists an $\varepsilon_0\in(0, \mathfrak{i}_{o})$ such that for any $t\in(0, \varepsilon_0)$,
\beqn  \frac{1}{2}e^{-\tau(y)}t^{n-1} \leq \hat{\sigma}_o(t,y) \leq 2e^{-\tau(y)}t^{n-1}.\label{sigma-bound}\eeqn Consequently,
\beqn A(S_t(o))=\int_{S_o(M)}\hat{\sigma}_o(t,y)d\nu_{o}(y) \leq 2\int_{S_o(M)}e^{-\tau(y)}t^{n-1}d\nu_{x_0}(y)=2\mathfrak{L_m}(o) t^{n-1},\eeqn and
$A(S_t(o))\geq \frac 12\mathfrak{L_m}(o) t^{n-1}$. Without loss of generality, we assume that $0<\varepsilon<R_1<R_2<\varepsilon_0<\mathfrak i_o$ (we choose sufficiently small $\varepsilon$, $R_1$, $R_2$ if needed). From these and (\ref{J1}), one obtains
\beq \mathfrak{J}_1\geq \frac 12\mathfrak{L_m}(o)\int_\varepsilon^{R_1}\frac 1tdt = \frac{1}{2} \left(\ln R_1- \ln\varepsilon\right)\mathfrak{L_m}(o)\rightarrow +\infty\ \ \ \  {\text as}\ \ \  \varepsilon\rightarrow 0^+,\label{J1*}\eeq
and
\beqn
\int_{ B_{\varepsilon}}r^{-p}d\mathfrak{m}= \int_0^\varepsilon t^{-p}A(S_t(o))dt\leq 2\mathfrak{L_m}(o) \int_0^\varepsilon t^{n-p-1}dt= \frac{2\varepsilon^{n-p}}{n-p}\mathfrak{L_m}(o).
\eeqn
Since $p\bar\beta=n-p>0$,  \beq \mathfrak{J}_4=\varepsilon^{-(n-p)}\int_{ B_{\varepsilon}(o)}r^{-p}d\mathfrak{m}\leq \frac{2}{n-p}\mathfrak{L_m}(o)< +\infty. \label{J4*}\eeq
On the other hand,  the integral on RHS of (\ref{Lp-Hardy}) is estimated by
\beq
\mathcal I_3&:=&\left|\int_{M}\left(1+\frac{n-1}{n-p}\zeta_{\kappa, h}(r)\right)r^{-p}|u|^p d\mathfrak{m}\right| \nonumber\\
&=& \left|\left(\int_{B_{\varepsilon}(o)}+\int_{B_{R_1}(o)\setminus B_{\varepsilon}(o)}+\int_{B_{R_2}(o)\setminus B_{R_1}(o)}\right)\left(1+\frac{n-1}{n-p}\zeta_{\kappa, h}(r)\right)r^{-p}|u|^p d\mathfrak{m}\right| \nonumber\\
&\geq & |A_2|-|A_1|-\mathfrak{J}_5, \label{I3}
\eeq
where $A_1, A_2$ and $\mathfrak J_5$ are respectively the first, second and third integrals in the second equality. Note that $\mathfrak J_5$ is finite and independent of $\varepsilon$. Since $\lim_{r\rightarrow 0^+}\zeta_{\kappa, h}(r))=0$, there exists a small number, also denoted by $\varepsilon_0$, in $(0, \mathfrak{i}_{o})$ such that $|\frac{n-1}{n-p}\zeta_{\kappa, h}(r)|\leq\eta$ for any small $0<\eta<1$ and $r\in(0, \varepsilon_0)$. Therefore,
\beqn |A_1|&\leq & (1+\eta)\varepsilon^{-p\bar\beta}\int_{B_{\varepsilon}(o)}r^{-p} d\mathfrak{m}=(1+\eta)\mathfrak{J}_4, \nonumber \\
|A_2|&\geq &  (1-\eta)\int_{B_{R_1}(o)\backslash B_{\varepsilon}(o)}r^{-p}|u|^p d\mathfrak{m}=(1-\eta)\mathfrak{J}_1.
\eeqn
Plugging these in (\ref{I3}) gives
$\mathcal I_3\geq(1-\eta)\mathfrak{J}_1-(1+\eta)\mathfrak{J}_4-\mathfrak{J}_5.$
From this, (\ref{J1*})--(\ref{J4*}), $\mathfrak J_5<+\infty$ and (\ref{Lp-Hardy}), one obtains
\beqn
C &\leq &\frac{(\mathcal I_1)^{\frac{1}{p}}(\mathcal I_2)^{\frac{1}{p'}}}{\mathcal I_3}
\leq \lim\limits_{\varepsilon\rightarrow0^+}\frac{({\bar\beta}^p\mathfrak{J}_1+\mathfrak{J}_2)^{\frac{1}{p}}(\mathfrak{J}_1+\mathfrak{J}_3
+\mathfrak{J}_4)^{\frac{1}{p'}}}{(1-\eta)\mathfrak{J}_1-(1+\eta)\mathfrak{J}_4-\mathfrak{J}_5}\nonumber \\
&=& \lim\limits_{\varepsilon\rightarrow 0^+}\frac{\bar\beta\mathfrak{J}_1\left(1+\frac{\mathfrak{J}_2}{{\bar\beta}^p\mathfrak{J}_1}\right)^{\frac{1}{p}}
\left(1+\frac{\mathfrak{J}_3}{\mathfrak{J}_1}+\frac{\mathfrak{J}_4}{\mathfrak{J}_1}\right)^{\frac{1}{p'}}}{(1-\eta)\mathfrak{J}_1-(1+\eta)\mathfrak{J}_4-\mathfrak{J}_5}
=\frac{\bar\beta}{1-\eta}.
\eeqn
Letting $\eta\rightarrow0^+$ leads to $C \leq \lim_{\eta\rightarrow0^+}\frac{\bar\beta}{1-\eta}=\bar\beta$,
which contradicts the assumption, and hence the claim is proved. Moreover,  by Theorem \ref{thm11}, the equality in (\ref{Hardy}) holds if and only if
\beqn d|u|=-r^{-1}|u|dr, \ \ \   \max\left\{F^{\ast}(\pm dr)\right\}= F^{\ast}(dr), \ \  \Delta r=(n-1)\big(\mathfrak{ct}_\kappa(r)-h\big).  \eeqn
Solving the first equation yields $|u|=cr^{-1}$ for some $c>0$. However $u\notin C_0^\infty(M)$ since $u$ is singular at $o$. Hence there is not a function in $C_0^\infty(M)$ satisfying the equality in (\ref{Hardy}).
\end{proof}

\begin{proof} [Proof of Theorem \ref{thm14}] For any $x\in M$, let $r_x=d_F(x, \cdot)$. Then $\mathbf K(\nabla r_x)\leq \kappa$ and $\mathbf S(\nabla r_x)\leq (n-1)h$ for some $\kappa, h\in \mathbb R$ by the assumption. From Proposition \ref{prop45}, we have (\ref{Hardy}) in which $r$ is replaced by $r_x$ for any $u\in C_0^\infty(M)$, which implies (\ref{H-Ineq}). The sharpness of  $\frac{n-p}p$ follows from $\Lambda=1$ and Proposition \ref{prop45}.
 \end{proof}

 \section{Examples}

 In this section, we give some special Finsler measure spaces on which the $L^p$-uncertainty principles hold.

 \begin{exam} {\rm The well-known Funk metric $\mathcal F$ is defined by $\mathcal F=\alpha+\beta$ on a unit open ball $\mathbb B^n$ in $\mathbb R^n$, where
$\alpha$ and $\beta$ are respectively a Riemannian metric and a 1-form on $\mathbb B^n$ given by
 $$\alpha=\frac {\sqrt{(1-|x|^2)|y|^2+\langle x , y\rangle^2}}{1-|x|^2}, \ \ \  \ \ \beta=\frac{\langle x , y\rangle}{1-|x|^2}, \ \  x\in \mathbb B^n,
 \ y\in T_x\mathbb B^n, $$ where $\langle , \rangle$ is the standard Euclidean inner product in $\mathbb R^n$ and $|\cdot|$ is the norm induced by $\langle , \rangle$.
  It is easy to see that $\beta$ is closed. Thus $F$ is projectively flat. Moreover, $\mathcal F$ has constant flag curvature $\mathbf K=\kappa=-\frac 14$ and constant S-curvature $\mathbf S=\frac {n+1}{2}$, i.e., $h=\frac{n+1}{2(n-1)}$,  with respect to the measure $\mathfrak m_{BH}$. Further, $\mathcal F$ is a forward complete (but not backward complete) Randers metric with infinite reversibility (\cite{Sh}, \cite{HKZ}). This Funk space $(\mathbb B^n, \mathcal F, \mathfrak m_{BH})$ satisfies the curvature assumptions in Theorems \ref{thm12}--\ref{thm14}. Thus $L^p$-uncertainty principles hold on $(\mathbb B^n, \mathcal F, \mathfrak m_{BH})$. As an example, we shall check the $L^p$-Hardy inequality (see Theorem \ref{thm14}) on the Funk space $(\mathbb B^n, \mathcal F, \mathfrak m_{BH})$.

It is known that the distance function for $\mathcal F$ is $r:= d_F(o, x)=-\log (1-|x|)\in (0, +\infty)$ for any $x\in \mathbb B_1^n$ and $d\mathfrak m_{BH}=dx$ (Euclidean volume form) (\cite{Sh}, \cite{HKZ}). In polar coordinates $(t, \theta)$, where $t=|x|$, we have $t=1-e^{-r}\in (0, 1)$ and $d\mathfrak m_{BH}=t^{n-1}dt d\theta$. For any $0<a\in \mathbb R$, let \beqn u_a:=-e^{-ar}=- (1-t)^a. \eeqn   Note that $F^{\ast}(x, \xi)=|\xi|-\langle x, \xi\rangle$ for any $\xi\in T^{\ast}_x\mathbb B^n$.  We have
\beq  F^{\ast}(du_a)=a(1-t)^a, \ \ \ \  F^{\ast}(-du_a)=a(1-t)^{a-1}(1+t). \label{FF*1}\eeq
Clearly, $ F^{\ast}(du_a)\neq  F^{\ast}(-du_a)$ and $\max\{F^{\ast}(\pm du_a)\}=F^{\ast}(-du_a)=a(1-t)^{a-1}(1+t)$. It is easy to see that $u_a\in \mathcal W_0^{1, p}$ (resp. $W_0^{1, p}(M)$) for $a>(p-1)/p$ (cf. (\ref{A}) below).  Next we check (\ref{Hardy}) in the case of $p=2$ for simplicity and hence $n>2$ by the assumption on $p$ in Theorem \ref{thm14}. In this case, we have $a>1/2$. The verification is similar for general $p>1$.

In fact, let $\omega_{n-1}$ be the area of unit sphere $\mathbb S_1^{n-1}$. Since $a>1/2$, we have
\beq A:&=&\int_{\mathbb B_1^n} \max\{F^{\ast 2}(\pm du_a)\}d\mathfrak m_{BH}=a^2\omega_{n-1}\int_0^1(1-t)^{2a-2}(1+t)^2 t^{n-1}dt \nonumber \\
 &=& a^2\omega_{n-1}\big\{\mathcal B(n, 2a-1)+ 2\mathcal B(n+1, 2a-1)+\mathcal B(n+2, 2a-1)\big\}\nonumber \\
 &=& a^2\omega_{n-1}\frac{4(n+a)^2-2a}{(n+2a)(n+2a-1)}\mathcal B(n, 2a-1),\label{A}
 % \nonumber \\ &=&  a^2\omega_{n-1}\frac{4(n+a)^2-2a}{(n+2a)(n+2a-1)}\cdot\frac{n-1}{n+2a-1}\cdot\frac{n-2}{n+2a-2}\mathcal B(n-2, 2a)
 \eeq where  $\mathcal B(\mathfrak p, \mathfrak q)$ is the Beta function with $\mathcal B(\mathfrak p, \mathfrak q)=\mathcal B(\mathfrak q, \mathfrak p)=\frac{\mathfrak q-1}{\mathfrak p+\mathfrak q-1}\mathcal B(\mathfrak p, \mathfrak q-1)$ for $\mathfrak p>0$ and $\mathfrak q>1$. Similarly,
\beqn B:&=&\int_{\mathbb B_1^n}r^{-2}u_a^2d\mathfrak m_{BH}=\omega_{n-1}\int_0^1\big(\log(1-t)\big)^{-2}(1-t)^{2a}t^{n-1}dt \nonumber \\
&\geq & \omega_{n-1}\int_0^1t^{n-3}(1-t)^{2a+2}dt=\omega_{n-1}\mathcal B(n-2, 2a+3),  \eeqn where we used the inequality $-\log(1-t)\leq t/(1-t)$ for any $t\in (0, 1)$. Since
$$\zeta(r):= \zeta_{-\frac 14, \frac{n+1}{2(n-1)}}(r)=r\left(\frac 12\coth\frac r2-\frac{n+1}{2(n-1)}\right)-1=r\left(\frac 1{1-e^{-r}}-\frac n{n-1}\right)-1, $$
we have $$1+\frac{n-1}{n-2}\zeta(r)=\frac 1{n-2}\left\{-1-(n-1)t^{-1}\log(1-t)+n\log(1-t)\right\}.  $$
From this,  the integral on RHS of (\ref{Hardy}) is given by
\beq C:&=&\int_{\mathbb B_1^n}\left(1+\frac{n-1}{n-2}\zeta(r)\right)r^{-2}u_a^2d\mathfrak m_{BH} \nonumber \\
&=& \frac {\omega_{n-1}}{n-2}\Big\{-\int_0^1t^{n-1}(1-t)^{2a}\big(\log(1-t)\big)^{-2}dt-(n-1)\int_0^1t^{n-2}(1-t)^{2a}\big(\log(1-t)\big)^{-1}dt\nonumber \\
& & +n\int_0^1t^{n-1}(1-t)^{2a}\big(\log(1-t)\big)^{-1}dt\Big\}.\label{C} \eeq
Integrating the first term by parts and using $n>2$ yield
\beqn -\int_0^1t^{n-1}(1-t)^{2a}\big(\log(1-t)\big)^{-2}dt=\int_0^1\Big[(n-1)-(n+2a)t\Big]t^{n-2}(1-t)^{2a}\big(\log(1-t)\big)^{-1}dt.\eeqn
Plugging this in (\ref{C}) and using the inequality $\log(1-t)\leq -t$ for any $t\in (0, 1)$  yields
\beqn 0<C&=&-\frac {2a\omega_{n-1}}{n-2}\int_0^1t^{n-1}(1-t)^{2a}\big(\log(1-t)\big)^{-1}dt\nonumber \\
&\leq &\frac {2a\omega_{n-1}}{n-2}\int_0^1t^{n-2}(1-t)^{2a}dt=\frac {2a\omega_{n-1}}{n-2}\mathcal B(n-1, 2a+1).  \eeqn
Consequently, one obtains
\beqn \frac {AB}{C^2}\geq \frac {(n-2)^2}4 \cdot \frac{4(n+a)^2-2a}{(n+2a)(n+2a-1)}\cdot \frac {\mathcal B(n, 2a-1)\mathcal B(n-2, 2a+3)}{\mathcal B^2(n-1, 2a+1)}.\eeqn
Observe that
\beqn & &\frac{4(n+a)^2-2a}{(n+2a)(n+2a-1)}\cdot \frac {\mathcal B(n, 2a-1)\mathcal B(n-2, 2a+3)}{\mathcal B^2(n-1, 2a+1)}\nonumber \\
%& &=\frac{4(n+a)^2-2a}{(n+2a)(n+2a-1)}\cdot \frac{n-1}{n+2a-2}\cdot \frac{n+2a-2}{2a-1}\cdot \frac{n+2a-1}{2a}\cdot \frac{n+2a+1}{n-2} \cdot \frac{2a+2}{n+2a+1} \cdot \frac{2a+1}{n+2a}\nonumber \\
&  &= \frac{(a+1)(2a+1)}{a(2a-1)}\cdot \frac{(n-1)\big[4(n+a)^2-2a\big]}{(n-2)(n+2a)^2}>1 \eeqn
for any $a>1/2$, in which the last inequality holds due to each factor in the numerator is greater than the corresponding one in the denominator. Thus $AB\geq \frac{(n-2)^2}4C^2$. Namely, the $L^2$-Hardy inequality in the form (\ref{Hardy}) holds for $u_a=-e^{ar}$ ($a>\frac 12$) on Funk spaces $(M, \mathcal F, \mathfrak m_{BH})$ although the $L^2$--Hardy inequality on $(M, \mathcal F, \mathfrak m_{BH})$ in the form as in Example 6.2 of \cite{HKZ} fails. } \end{exam}

\begin{exam} {\rm As a modification of the Funk metric $\mathcal F$, Kaj\'ant\'o-Krist\'aly considered a family of the interpolated Funk-type metics $F_k=\alpha+\beta_k$ on $\mathbb B^n\subset \mathbb R^n$, where $\alpha$ is the one as before and $\beta_k: =\frac{k\langle x, y\rangle}{1-|x|^2}$ for $k\in [0, 1)$ (\cite{KK}). It is easy to see that $F_k$ is forward complete with finite reversibility $\Lambda=\frac{1+k}{1-k}$. In \cite{KK}, the authors proved that the flag curvature and S-curvature of $F_k$ with respect to $\mathfrak m_{BH}$ are estimated as
\beqn -\frac 1{(1-k)^2}<\mathbf K_k<-\frac 1{(1+k)^2} \ \ {\rm{and}}\ \ 0 <\mathbf S_k <\frac{(n+1)k}{2(1-k^2)}\eeqn up to an isometry.  If $k\in (\frac{2(n-1)}{3n-1}, 1)$, then $(\mathbb B^n, F_k, \mathfrak m_{BH})$ satisfies the assumptions of Theorem \ref{thm12} and \ref{thm14}. Consequently, $L^p$--uncertainty principles hold in Theorem \ref{thm12} and \ref{thm14} on such spaces $(\mathbb B^n, F_k, \mathfrak m_{BH})$. }\end{exam}

 \section{Appendix: The spectral gap estimates for Finsler $p$-Laplacian}

 In this section, we further seek applications of Theorem \ref{thm11} as mentioned in the introduction. We first prove the $L^p$-Caccioppoli inequality (or the reverse $L^p$-Poincar\'e inequality) on Finsler measure spaces. As an application, we obtain the first Dirichlet eigenvalue estimate for Finsler $p$-Laplacian. In fact,  the Caccioppoli inequality plays an important role in the energy estimate or gradient estimate for harmonic functions (cf. \cite{HL}).

\begin{thm}[Caccioppoli inequality]\label{thm15}
Let $(M, F, \mathfrak{m})$ be an $n (\geq 2)$-dimensional forward complete and noncompact Finsler measure space and $\Omega$ be a domain in $M$. For any $x\in \Omega$ and $1<p<n$,  let $r$ be the distance function on $(M, F)$.  If $\Delta r\leq0$ in $\Omega$ (in the weak sense), then
\beq
	\int_{\Omega}\max\{F^{\ast p}(\pm du)\}d\mathfrak{m} \geq \left(\frac{p-1}{p}\right)^p \int_{\Omega} \frac{|u|^p}{r^p} d\mathfrak{m}. \ \ \ \forall u\in C^{\infty}_0(\Omega), \label{cocci}\eeq
and the constant $\left(\frac{p-1}{p}\right)^p$ is sharp if $F^{\ast}(-dr)\leq F^{\ast}(dr)$. Further, the equality holds if and only if
$ |u|=cr^{\frac{p-1}p}$ for some $c>0$ and $r$ is a Finsler harmonic function with  $F^{\ast}(-dr)\leq F^{\ast}(dr)$.

In particular, if  $\Lambda<\infty$, $\Omega$ is bounded  and $\Delta\mathbf r\leq0$ in $\Omega$ (in the weak sense) for $\mathbf r(\cdot):= d_F( \partial\Omega, \cdot)$, then the first Dirichlet eigenvalue of $p$-Laplacian on $\Omega$ is given by
\begin{equation*}
	\lambda_{1,p} (\Omega)=\inf\limits_{u\in C_0^\infty(\Omega)\setminus\{0\}}\frac{\int_{\Omega}F^{{\ast}p}(du)d\mathfrak{m}}{\int_\Omega|u|^pd\mathfrak{m}}\geq \left(\frac{p-1}{p}\right)^p \left(\frac{1}{\Lambda R_\Omega}\right)^p,
\end{equation*}
where $R_\Omega=\sup_{x\in \Omega}\mathbf r(x)$ is the inradius of $\Omega$.
\end{thm}
\begin{proof}  Let us choose  the functions $G(s)=-\left(\frac{p-1}{p}\right)^{p-1}s^{1-p}$ and $H(t)=\frac 1p |t|^p$. Then $G(s)$ and $H(t)$ satisfy the assumptions in  Theorem \ref{thm11} as before. Applying Theorem \ref{thm11}(1) on $\Omega$ and using $\Delta r\leq0$ in $\Omega$,  we have
\beqn \int_{\Omega}\max\{F^{\ast p}(\pm du)\}d\mathfrak{m}
% \geq \int_{\Omega}\left(\frac{(p-1)^p}{p^{p-1}}r^{-p}-\left(\frac{p-1}{p}\right)^{p-1}r^{1-p}\Delta r\right)|u|^pd\mathfrak{m}-(p-1)\int_{\Omega}\left(\frac{p-1}{p}\right)^{p}r^{-p}|u|^{p}d\mathfrak{m}\\
&\geq & \int_{\Omega}\frac{(p-1)^p}{p^{p-1}}r^{-p}|u|^p d\mathfrak{m}-(p-1)\int_{\Omega}\left(\frac{p-1}{p}\right)^{p}r^{-p}|u|^{p}d\mathfrak{m}\nonumber \\
&= &\left(\frac{p-1}{p}\right)^{p}\int_{\Omega}r^{-p}|u|^{p}d\mathfrak{m},\eeqn in which the equality holds if and only if
\beqn d|u|=\frac {p-1}{p}|u|r^{-1}dr, \ \ \  \max\{F^{\ast p}(\pm dr)\}=F^{\ast}(dr), \ \ \ \Delta r=0 \ ({\rm{in\  the\  weak\  sense}}).\eeqn Solving the first equation yields $|u|=cr^{\frac{p-1}p}$ for some $c>0$. The second and third equalities mean that $r$ is a Finsler harmonic function with  $F^{\ast}(-dr)\leq F^{\ast}(dr)$.

Next we assume by contradiction that there exists a constant $C>\left(\frac{p-1}{p}\right)^p$ such that
\beqn
	\int_{\Omega}\max\{F^{\ast p}(\pm du)\}d\mathfrak{m} \geq C \int_{\Omega} \frac{|u|^p}{r^p} d\mathfrak{m}
\eeqn for any $u\in C^{\infty}_0(\Omega)$.
If  $F^{\ast}(-dr)\leq F^{\ast}(dr)$, then taking $u=cr^{1-\frac 1{p}}$ for some $c\in\mathbb{R}$ in the above inequality yields
 $C\leq\left(\frac{p-1}{p}\right)^p$, which contradicts the assumption. Finally, the estimate of $\lambda_{1, p}(\Omega)$ follows from (\ref{cocci}) and $\max\{F^{{\ast}p}(\pm du)\}\leq \Lambda^p F^{{\ast}p}(du)$.
\end{proof}

The well known McKean's spectral gap estimate for the Laplacian says that the first eigenvalue $\lambda_1\geq \frac{(n-1)^2}4k$ on a complete noncompact Riemanninan manifold with the sectional curvature $K\leq -k (0<k\in \mathbb R)$ (\cite{Mc}). Wu-Xin generalized this estimate to the case of Finsler Laplacian  (\cite{WX}). Now we present McKean's spectral gap estimate for Finsler $p(>1)$-Laplacian, which extends Wu-Xin's result in a different way.
\begin{thm}\label{thm16}
Let $(M,F,\mathfrak{m})$ be an $n (\geq 2)$-dimensional forward complete and noncompact Finsler measure space. Assume that $\mathbf K\leq -\kappa^2$ and $\mathbf S\leq(n-1)h$ for some constants $\kappa> 0$ and $h <\kappa$. Then, for any  $p>1$ and $u\in C^{\infty}_0(M)\backslash\{0\}$, we have
\beqn
	\int_{M}\max\left\{F^{\ast p}(\pm du)\right\}d\mathfrak{m} \geq \left(\frac{2(n-1)(\kappa-h)}{p}\right)^p \frac{ \left(\int_{M} |\mathfrak{s}_c(u)|^p d\mathfrak{m}\right)^p } {\left(\int_{M} |\mathfrak{s}_c^{p-2}(u)\mathfrak{s}_c(2u)|^{p'} d\mathfrak{m}\right)^{p-1}}, \label{ineq-eigen}
\eeqn
for any $c\in \mathbb{R}$. In particular, if $\Lambda<\infty$, then we have McKean's spectral gap estimate for the Finsler $p$-Laplacian
 \beqn
	\lambda_{1,p} \geq  \left(\frac{( n-1)(\kappa-h)}{p\Lambda}\right)^p.
\eeqn\end{thm}
\begin{proof}
 Let $$G(s) \equiv 1, \ \ \ \ H(t)=|\mathfrak{s}_c(t)|^p$$ for any $s>0$ and $c, t\in \mathbb R$.
Then $G'(s)=0$ and $|H'(t)|=\frac{p}{2}|\mathfrak{s}_c(t)|^{p-2}|\mathfrak{s}_c(2t)|$. Obviously,
 $G(s)$ and $H(t)$ satisfy the requirements of Theorem \ref{thm11}. By Theorem \ref{thm21}, we have $\Delta r\geq(n-1)(\kappa\coth(\kappa r)-h)\geq(n-1)(\kappa-h)$. From this and Theorem \ref{thm11}(2), one obtains
\beqn
	\int_{M}\max\{F^{\ast p}(\pm du)\}d\mathfrak{m} &\geq &
\frac{\left(\int_{M}(n-1)(\kappa-h)|\mathfrak{s}_c(u)|^pd\mathfrak{m}\right)^p}
{\left(\int_{M}(\frac{p}{2})^{p'}|\mathfrak{s}^{p-2}_c(u)\mathfrak{s}_c(2u)|^{p'}d\mathfrak{m}\right)^{p-1}}\\
&=&\left(\frac{2(n-1)(\kappa-h)}{p}\right)^{p}\frac{(\int_{M}|\mathfrak{s}_c(u)|^pd\mathfrak{m})^p}
{\left(\int_{M}|\mathfrak{s}^{p-2}_c(u)\mathfrak{s}_c(2u)|^{p'}d\mathfrak{m}\right)^{p-1}}.
\eeqn
In particular, when $c=0$, we have $\mathfrak{s}_c(u)=u$ and hence
\beqn
	\Lambda^p\int_MF^{\ast p}(du)dm\geq \int_{M}\max\{F^{\ast p}(\pm du)\}d\mathfrak{m} \geq
\left(\frac{(n-1)(\kappa-h)}{p}\right)^{p}\int_{M}|u|^pd\mathfrak{m},
\eeqn
which implies the conclusion. \end{proof}

\bigskip

% \noindent {\bf Acknowledgements} The project is partially supported by NNSFC (Nos.12471044, 12071423) and Zhejiang Provincial NSFC (No. LZ26A010004).
%The authors would like to thank Alexandru Krist\'aly and Wei Zhao for their suggestions on the first draft of this paper.

\bigskip

\noindent {\bf Data Availability} Data sharing is not applicable to this article as no data sets were generated or analyzed during the current study.
\bigskip

\noindent {\bf Declarations}

\noindent {\bf Conflict of interest} \  The authors declare that they have no conflict of interest.

\end{document}